\documentclass[11pt,reqno]{amsart}
\usepackage{amssymb}
\usepackage{enumerate}
\usepackage{amsmath,amssymb,amsfonts,amsthm,
	latexsym, amscd, epsfig, epic,color}
\usepackage{epstopdf}
\usepackage{mathtools}
\usepackage{youngtab}
\usepackage{extarrows}
\usepackage{enumitem}
\usepackage{enumerate}
\usepackage{mathrsfs}
\usepackage{eucal}%    caligraphic-euler fonts: \mathcal{ }
\usepackage{eufrak}%   frak-euler        fonts: \mathfrak{ }
\usepackage{xspace}
\usepackage{a4wide}
\usepackage{colordvi}
\usepackage{comment}
\usepackage{hyperref,eucal}
\usepackage{tabularx}
\usepackage{tikz}\usetikzlibrary{matrix,arrows}
\usepackage{ytableau} \usepackage{chemfig}
\usepackage{float}
\usepackage{youngtab}
\usepackage{mathrsfs}
\usepackage{caption}
\usepackage{epsfig}
\usepackage{stfloats}
\usepackage{booktabs}
\usepackage{changepage}
\usepackage{graphicx}
\theoremstyle{plain}

\theoremstyle{plain}

\newtheorem{theorem}{Theorem}[section]

\newtheorem{lemma}{Lemma}[section]
\newtheorem{proposition}{Proposition}[section]

\newtheorem{definition}{Definition}[section]

\theoremstyle{example}
\newtheorem{example}{Example}[section]

\theoremstyle{openproblem}

\def\T{\CMcal{T}}

\def\S{\CMcal{S}}

\def\maj{\mathsf{maj}}

\makeatletter
\@namedef{subjclassname@2020}{\textup{2020} Mathematics Subject Classification}
\makeatother

\newcommand{\bh}{\mathbf{h}}

\newcommand{\Sym}{\mathrm{Sym}}
\newcommand{\HL}{\mathrm{HL}}

\newcommand{\slashp}{\mid}
\newcommand{\suchthat}{\,:\,}

\newcommand{\bx}{{\mathbf x}}
\newcommand{\bX}{{\mathbf X}}
\newcommand{\bm}{{\mathbf m}}
\newcommand{\be}{{\mathbf e}}
\newcommand{\bs}{{\mathbf s}}
\newcommand{\bp}{{\mathbf p}}
\newcommand{\bS}{{\mathbf S}}
\newcommand{\bM}{{\mathbf M}}
\newcommand{\bP}{{\mathbf P}}
\newcommand{\bL}{{\mathbf L}}

\newcommand{\bPP}{\boldsymbol{\CMcal{P}}}

\theoremstyle{remark}

\newtheorem{remark}{Remark}
\numberwithin{equation}{section}
\numberwithin{figure}{section}
\title[Hall-Littlewood functions in noncommuting variables]{Hall-Littlewood functions in noncommuting variables}

\author{Soojin Cho}
\address{Department of Mathematics, Ajou University, Suwon 16499, Republic of Korea}
\email{chosj@ajou.ac.kr}

\author{Emma Yu Jin}
\address{School of Mathematical Sciences, Xiamen University, Xiamen 361005, China}
\email{yjin@xmu.edu.cn}


\date{\today}

\begin{document}

\begin{abstract}
	In 2022 Aliniaeifard, Li, and van Willigenburg defined Schur functions in the algebra of symmetric functions in noncommuting variables (NCSym), answering an open question posed by Rosas and Sagan in 2004. These Schur functions are not monomial positive, since they are defined via a noncommutative analogue of the Jacobi-Trudi determinant.
	
	We introduce Hall-Littlewood functions $\bP_{\pi}(\bx;t)$ indexed by set partitions $\pi$ in noncommuting variables $\bx=(\bx_1,\bx_2,\ldots)$, and define Schur functions in noncommuting variables to be $\bs_{\pi}(\bx)=\bP_{\pi}(\bx;0)$. Just as the classical Hall-Littlewood functions in commuting variables, our Hall-Littlewood functions $\bP_{\pi}(\bx;t)$ in NCSym interpolates between monomial symmetric functions $\bm_{\pi}(\bx)$ and Schur functions $\bs_{\pi}(\bx)$.
	
	We prove that the set of Hall-Littlewood functions $\{\bP_{\pi}(\bx;t)\}$ for all set partitions $\pi$ of $[n]$ forms a $\mathbb{Q}[t]$-basis of NCSym of homogeneous degree $n$, and that this basis is invariant under any permutation acting on set partitions. These Hall-Littlewood functions in NCSym map to classical Hall-Littlewood functions under commutation, up to a scalar factor. 
	
	We also show that the Hall-Littlewood functions $\bP_{\pi}(\bx;t)$ naturally refine the lifted Hall-Littlewood functions in NCSym. Specifically, the Schur functions $\bs_{\pi}(\bx)$ are monomial positive and refine the lifted Schur function introduced by Rosas and Sagan. Moreover, we introduce a star product of two polynomials in NCSym and develop the star-multiplication rule for a lifted and a non-lifted Hall-Littlewood functions in NCSym. This rule is a noncommutative analogue of the product rule for two Hall-Littlewood functions and, in particular, of the Littlewood-Richardson rule. Finally, our approach extends to the algebra of quasisymmetric functions in noncommuting variables (NCQSym) indexed by set compositions.

\end{abstract}
	\maketitle
	%%%%%%%%%%%%%%%%%%%%%%%%%%%%%%%%%%%%    
	\section{Introduction and main results}
	The Hall-Littlewood (HL) functions $P_{\lambda}(X;t)$ indexed by partitions $\lambda$ are symmetric and orthogonal polynomials in infinitely many variables $X=(x_1,x_2,\ldots)$ with coefficients in the ring $\mathbb{Z}[t]$ of polynomials \cite[Chapter III]{Mac95}. They were originally discovered by Philip Hall in terms of Hall algebra \cite{Hall} and later thoroughly developed by D.E. Littlewood \cite{Littlewood:61}. The Hall-Littlewood function $P_{\lambda}(X;t)$ is a specialization of symmetric Macdonald polynomials $P_{\lambda}(X;q,t)$ at $q=0$, which interpolates between the monomial symmetric functions $m_{\lambda}(X)=P_{\lambda}(X;1)$ and the Schur function $s_{\lambda}=P_{\lambda}(X;0)$. 
	
	The HL functions $P_{\lambda}(X;t)$ form an orthogonal $\mathbb{Z}[t]$-basis in the algebra of symmetric functions in commuting variables (Sym). The objective of this paper is to develop a $\mathbb{Q}[t]$-basis of HL functions in noncommuting variables. Let us briefly review the history of symmetric functions in noncommuting variables (NCSym).
	
	The study of NCSym was initiated in 1936 by Wolf \cite{Wolf}, who produced an analogue of fundamental theorem of symmetric functions in this setting. The noncommutative idea remained inactive for over three decades until Bergman and Cohn revisited Wolf's work and generalized her result \cite{BC:69}. The major breakthrough came only in 2004 when Rosas and Sagan constructed noncommutative analogues of classical symmetric function concepts such as the monomial, power-sum, elementary and complete homogeneous symmetric functions, each of which can project to its scaled commutative counterpart \cite{RS}.
	
	Rosas and Sagan posed a question \cite[Section 9]{RS} on the Schur basis of NCSym that behaves like the classical Schur functions, and that specializes to Schur functions under the projection map. 
	In response, in 2022 Aliniaeifard, Li and van Willigenburg \cite{alw:22} resolved this question by providing the first Schur basis for NCSym indexed by set partitions that exhibits many properties analogous to the classical Schur functions. However, unlike the Schur functions for Sym, this Schur function is not a positive sum of monomial symmetric functions (monomial positive or $m$-positive). Moreover, it is not clear how to generalize this Schur function to an HL function for NCSym.
	
	In this paper, we introduce HL functions $\bP_{\pi}(\bx;t)$ in noncommuting variables $\bx=(\bx_1,\bx_2,\ldots)$ indexed by all set partitions $\pi$, and the new Schur function is defined by $\bs_{\pi}(\bx)=\bP_{\pi}(\bx;0)$. Some key properties of $\bP_{\pi}(\bx;t)$ are summarized as follows:
	\begin{enumerate}[label=\Roman*.]
		\item The HL functions $\bP_{\pi}(\bx;t)$ form a $\mathbb{Q}[t]$-basis of NCSym, and that this basis is invariant under any permutation acting on set partitions (see Theorem \ref{T:basis}). These polynomials $\bP_{\pi}(\bx;t)$ map to the classical HL functions $P_{\lambda}(X;t)$ under the projection map, and they interpolate between Schur functions $\bs_{\pi}(\bx)$ and monomial symmetric functions $\bm_{\pi}(\bx)$, up to a scalar factor (see Proposition \ref{prop:basic}).\\
		\item The HL functions $\bP_{\pi}(\bx;t)$ naturally refine the lifted HL function indexed by a partition (see Theorem \ref{T:sumP1}). Specifically, the Schur functions $\bs_{\pi}(\bx)$ as a positive sum of monomial symmetric functions
		 refine the lifted Schur function defined by Rosas and Sagan \cite{RS}; the Schur $P$-function $\bP_{\pi}(\bx;-1)$ refines the lifted Schur $P$-function.\\
		\item The star product $\star$ of two symmetric functions in NCSym is introduced. Although 
		the usual product or the star product of an HL function $\bP_{\pi}(\bx;t)$ and a Schur function $\bs_{1}(\bx)$ is not even a positive sum of HL functions, it turns out that the star product of a lifted HL function and a non-lifted HL function can be expanded as a positive sum of HL functions (see Theorem \ref{T:pieri1}). This is a noncommutative version of the well-known product rule for two HL functions in Sym, and in particular, of the Littlewood-Richardson rule.
	\end{enumerate}
    In addition, our approach can be applied to define the $\mathbb{Z}[t]$-bases of quasisymmetric functions in noncommuting variables (NCQSym) indexed by set compositions (equivalently, ordered set partitions), and specializations such as quasi-Schur functions, thereby discovering the quasisymmetric Littlewood-Richardson rule in the noncommutative setting; see Theorems \ref{T:quasi-basis} -- \ref{T:pieri2}.
    
	The remainder of the paper is structured as follows. Section \ref{S:2} provides the relevant background on symmetric functions. In Sections \ref{S:3} and \ref{S:4}, we introduce the Hall-Littlewood functions in NCSym and prove the main results (I)--(III). In Sections \ref{S:5} and \ref{S:6} we present the preliminaries on quasisymmetric functions and our quasisymmetric analogue of the main results.

	%%%%%%%%%%%%%%%%%%%%%%%%%%%%%%%%%%   
	\section{Background on symmetric functions}\label{S:2}
	We begin with preliminaries on the combinatorial concepts \cite{alw:22,BZ:09} we need, followed by the introduction of the algebra and symmetric functions under study. 
	
	\subsection{Compositions, partitions and set partitions}
	A {\em composition} $\alpha$ of $n$ is a finite sequence $\alpha=(\alpha_1, \dots, \alpha_l)$ of positive integers such that $\sum_{i=1}^{l} \alpha_i=n$, denoted by $\alpha\vDash n$. Each $\alpha_i$ is called a part of $\alpha$, $\ell(\alpha)=l$ is called the length of $\alpha$ and $|\alpha|=n$ is called the size of $\alpha$. We denote by $0$ the unique composition of length and size $0$. If the parts of $\alpha$ are allowed to include $0$, then we call this a {\em weak composition}. If the parts of $\alpha$ appear in weakly decreasing order, then we call this a {\em partition} $\lambda$, denoted by $\lambda\vdash n$. 
	
	Let $\lambda(\alpha)$ be the partition obtained by rearranging the parts of composition $\alpha$ in weakly decreasing order, and let $\alpha^+$ be the composition after removing $0$ from weak composition $\alpha$. Given a composition $\alpha=(\alpha_1,\ldots,\alpha_{\ell(\alpha)})$, we write the corresponding partition as $\lambda(\alpha)=\langle 1^{m_1}\cdots k^{m_k}\rangle$ where $i$ appears exactly $m_i$ times in $\alpha$, and define
	\begin{align*}
		\alpha!=\alpha_1!\cdots \alpha_{\ell(\alpha)}! \quad\mbox{ and }\quad \alpha^!=m_1!\cdots m_k!.
	\end{align*}
    Given two partitions $\lambda$ and $\mu$ of $n$, we say that $\lambda$ dominates $\mu$, written as $\mu\le \lambda$ if $\mu_1+\cdots+\mu_k\le \lambda_1+\cdots+\lambda_k$ for all $k>1$. For $k$ larger than the length of the partition, we extend the partition by adding zeros.
       
	\begin{example}
		If $\alpha=(2,1,3,2,2)$, then $\ell(\alpha)=5$, $|\alpha|=10$ and $\lambda(\alpha)=(3,2,2,2,1)$. Note that $\alpha!=2!1!3!2!2!=48$ and $\alpha^!=3!=6$. %and $\beta=(3,5,2)\succcurlyeq (1,2,1,3,2)$ as  $\mathrm{SET}(\beta)=\{3,8\}\subseteq \mathrm{SET}(\alpha)=\{2,3,6,8\}$.
	\end{example}
   A \emph{set partition} $\pi = \pi_1 / \cdots / \pi_l$ of $[n]$ is a set of disjoint nonempty subsets $\pi_1, \ldots, \pi_l$ of $[n]$ such that $\cup _{i=1}^{l} \pi_i = [n]$, denoted by $\pi \vdash [n]$. Each $\pi_i$ is called a {\em block} of $\pi$, $\ell(\pi)=l$ is called the \emph{length} of $\pi$, and $|\pi|=n$ is called the \emph{size} of $\pi$. Note that every set partition $\pi$ determines a partition $\lambda (\pi) = (\lambda(\pi) _1, \ldots , \lambda (\pi)_{\ell(\pi)})=\lambda(|\pi_1|, \dots, |\pi_l|)$, obtained by listing the cardinalities of the blocks of $\pi$ in weakly decreasing order. 
   We write 
   \begin{align*}
   	\pi!=\lambda(\pi)!\,\mbox{ and }\,\pi^!=\lambda(\pi)^!.
   \end{align*}
The {\em standard} representation of $\pi$ is obtained by arranging the blocks of $\pi$ so that, when read from left to right,
   \begin{enumerate}
   	\item the block sizes are weakly decreasing,
   	\item the smallest elements of blocks of equal size are strictly increasing, and
   	\item the elements within each block are strictly increasing.
   \end{enumerate}
   We then let  $\delta_{\pi}$ be the permutation written in one-line notation that is obtained by removing the bars between the blocks of $\pi$. We usually list the blocks in the above way, and denote by $\emptyset$ the unique set partition of length and size 0. For convenience, we always omit parentheses and commas when writing compositions and set partitions.  
   
   Given two set partitions $\pi \vdash [n]$ and $\sigma = \sigma _1 / \cdots / \sigma _{\ell(\sigma)} \vdash [m]$, we say that their \emph{slash product} $\pi \slashp \sigma$ is a set partition of $[n+m]$ consisting of blocks in $\pi$ and shifted blocks $\sigma _i+n = \{ s+n \suchthat s\in \sigma _i\}$ for $1\leq i \leq \ell(\sigma)$. Returning to compositions, given a composition $\alpha = (\alpha_1, \ldots, \alpha_{\ell(\alpha)})\vDash n$, its corresponding set partition $[\alpha]\vdash [n]$ is
   	\begin{align}\label{E:setalp}
   		[\alpha]&=[\alpha_1]\,\vert\, [\alpha_2]\,\vert \cdots\vert\, [\alpha_{\ell(\alpha)}],
   	\end{align}
   	where $[\alpha_i]=\{1,\ldots,\alpha_i\}$. Let $\mathfrak{S}_n$ be the set of permutations of $[n]$. A permutation $\delta\in \mathfrak{S}_n$ acts on a subset $B=\{i_1, \dots, i_m\}\subseteq [n]$ by $\delta B=\{ \delta(i_1), \dots, \delta(i_m)\}$. Hence,  $w\pi=w\pi_1/\cdots /w \pi_{\ell(\pi)}$ for $\pi=\pi_1/\cdots/\pi_{\ell(\pi)}\vdash [n]$. Consequently, any set partition $\pi$ with $\lambda(\pi)=\lambda$ can be expressed as 
   	\begin{align*}
   		\pi=\delta_{\pi}[\lambda].
   	\end{align*}
   	We also say that two set partitions $\pi, \sigma \vdash [n]$ satisfy $\pi\le \sigma$ if $\sigma$ is obtained from $\pi$ by merging blocks of $\pi$. Then this gives a lattice and the greatest lower bound of two set partitions is denoted by $\wedge$. 
   \begin{example}
   	If $\pi=24/1/3\vdash [4]$ and $\sigma=134/25\vdash [5]$, then $\pi\slashp\sigma=578/24/69/1/3\vdash [9]$. If $\alpha=(1,3,2,1,2)$, then $[\alpha]=[1]\slashp [3]\slashp [2]\slashp [1]\slashp [2]=234/56/89/1/7$.
   	
   		If $\pi= 2\,3\,5\,6\,7/1\,4\,8 /9\,11\,12 /10\,13 / 14$, then $\lambda(\pi)=\lambda=(5,3,3,2,1)$ and $\pi=\delta_{\pi}[\lambda]$, where $$\delta_{\pi}=2\,\,3\,\,5\,\,6\,\,7\,\,1\,\,4\,\,8\,\,9\,\, 11\,\,12\,\,10\,\,13\,\,14\in\mathfrak{S}_{14}$$ 
       written in one-line notation.
   \end{example}
    
  \subsection{Symmetric functions in commuting variables} 
  The graded Hopf algebra of symmetric functions, Sym, 
  \begin{align*}
  	\mathrm{Sym}=\mathrm{Sym}^0 \oplus \mathrm{Sym}^1 \oplus \cdots \subseteq \mathbb{Q}[[X]],
  \end{align*}
  where $\mathbb{Q}[[X]]=\mathbb{Q}[[x_1, x_2, \dots]]$ is the algebra of formal power series in the commuting variables $X=(x_1, x_2, \dots)$ over the rational field $\mathbb{Q}$, $\mathrm{Sym}^0=\mathrm{span}\{1\}$ and the $n$th graded piece for $n\ge 1$ has the bases 
  \begin{alignat*}{3}
  	\mathrm{Sym}^n&=\mathrm{span}\{m_{\lambda}:\lambda\vdash n\}&&=\mathrm{span}\{p_{\lambda}:\lambda\vdash n\}\\
  	&=\mathrm{span}\{e_{\lambda}:\lambda\vdash n\}&&=\mathrm{span}\{h_{\lambda}:\lambda\vdash n\}=\mathrm{span}\{s_{\lambda}:\lambda\vdash n\} &&
  \end{alignat*}
  where these bases are defined as follows, for a partition $\lambda=(\lambda_1,\ldots,\lambda_{\ell(\lambda)})$.
  
  The \emph{monomial symmetric function} $m_\lambda$ is defined by
  $$ m_\lambda=\sum x_{i_1}^{\lambda_1}\cdots x_{i_{\ell(\lambda)}}^{\lambda_{\ell(\lambda)}}\,,$$
  which is summed over all distinct monomials and the $i_j$ are also distinct. The \emph{elementary symmetric function} $e_\lambda$, the \emph{complete homogeneous symmetric function} $h_\lambda$ and the \emph{power-sum symmetric function} $p_\lambda$ are given by
  \begin{align}\label{E:prod}
  e_\lambda=e_{\lambda_1}\cdots e_{\lambda_{\ell(\lambda)}},\quad h_\lambda=h_{\lambda_1}\cdots h_{\lambda_{\ell(\lambda)}}\,\,\mbox{ and }\,\, p_\lambda=p_{\lambda_1}\cdots p_{\lambda_{\ell(\lambda)}}\,, 
  \end{align}
  where for each positive integer $r$, $$e_r=\sum_{ i_1<\cdots<i_r}x_{i_1}\cdots x_{i_r},\quad  h_r=\sum_{i_1\leq\cdots\leq i_r}x_{i_1}\cdots x_{i_r}\,\,\mbox{ and }\,\,p_r=\sum_{i=1}^{\infty} x_{i}^r\,.$$
  Given a partition $\lambda = (\lambda_1, \ldots, \lambda_{\ell(\lambda)})$, we say that its \emph{Young diagram}, also denoted by $\lambda$, is the array of left-justified boxes with $\lambda_i$ boxes in row $i$ from the top. Let $\lambda'$ be the transpose of $\lambda$; that is, the Young diagram of $\lambda'$ is obtained from that of $\lambda$ by reflecting across the main diagonal.
    
  A \emph{semistandard Young tableau} (SSYT) of shape $\lambda$, is a filling of $\lambda$ with positive integers so that each box is filled with exactly one positive integer, the entries along each row are weakly increasing from the left, and along each column strictly increasing from the top. Let $\mathrm{SSYT}(\lambda)$ denote the set of SSTYs of shape $\lambda$. The \emph{Schur function} $s_\lambda$ is defined to be the generating function of SSYTs:
  \begin{align}\label{E:stom}
  	s_{\lambda}=\sum_{T\in \mathrm{SSYT}(\lambda)}x^T=\sum_{T\in \mathrm{SSYT}(\lambda)}x_{c(T_1)}x_{c(T_2)}\cdots x_{c(T_{|\lambda|})}
  \end{align}
  where the boxes of $T$ are labelled by $T_1, T_2,\ldots, T_{|\lambda|}$, and $c(T_i)$ is the content of box $T_i$.  For a weak composition $\gamma$ of $|\lambda|$, 
  let $\mathrm{SSYT}(\lambda,\gamma)$ be the set of $T\in \mathrm{SSYT}(\lambda)$ such that  $\vert\{j\,\vert\,c(T_j)=i\}\vert =\gamma_i$ for all $i$, and let $K_{\lambda\mu}=|\mathrm{SSYT}(\lambda,\mu)|$ be the {\em Kostka number}. Then
  \begin{align}\label{E:stom2}
  	s_{\lambda}=\sum_{\mu\le\lambda}K_{\lambda\mu} m_{\mu}=m_{\lambda}+\sum_{\mu<\lambda}K_{\lambda\mu} m_{\mu}.
   \end{align}
  We now turn to a generalization of monomial symmetric functions and the Schur functions. The HL functions $P_{\lambda}(t)=P_{\lambda}(X;t)$ are defined as the unique $\mathbb{Z}[t]$-basis for the ring of symmetric functions satisfying the lower triangularity and the orthogonality: 
  \begin{align}\label{E:Pm11}
  	P_{\lambda}(t)=m_{\lambda}&+\sum_{\mu<\lambda}c_{\lambda\mu}(t)m_{\mu},
  \end{align}
  for some $c_{\lambda\mu}(t)\in \mathbb{Z}[t]$,
  and $\langle P_{\lambda},P_{\mu}\rangle_t=0$ if $\lambda\ne \mu$. The inner product $\langle\,\,,\,\,\rangle_t$ on the symmetric functions with values in $\mathbb{Q}(t)$ is defined by requiring that the bases of power-sum symmetric functions $\{p_{\lambda}\}$ be orthogonal:
  \begin{align}\label{E:inner11}
  	\langle p_{\lambda}, p_{\mu}\rangle_t=z_{\lambda}\,\delta_{\lambda\mu}\prod_{i\ge 1}\frac{1}{1-t^{\lambda_i}},
  \end{align}	
  where $z_{\lambda}=\prod_{i\ge 1}i^{m_i}m_i!$ and $m_i$ counts the multiplicity of part $i$ in the partition $\lambda$. For $t=1$, the inner product degenerates, and one defines $P_{\lambda}(1)$ by continuity in $t$, and that limit is $m_{\lambda}$. For $t=0$, we have $P_{\lambda}(0)=s_{\lambda}$ by (\ref{E:stom2}) and $\langle s_{\lambda}, s_{\mu}\rangle_0=\delta_{\lambda\mu}$. In other words, $c_{\lambda\mu}(1)=\delta_{\lambda\mu}$ and $c_{\lambda\mu}(0)=K_{\lambda\mu}$.
 
  Let us recall a tableau formula for HL functions $P_{\lambda}(t)$ in \cite[Chapter III, Equation $(5.11)'$]{Mac95}, as an extension of (\ref{E:stom2}). 
  
  Given two partitions $\lambda =   (\lambda_1, \ldots, \lambda_{\ell(\lambda)}) \vdash n$ and $\mu = (\mu_1, \ldots, \mu_{\ell(\mu)}) \vdash m$, we say that $\mu$ is \emph{contained} in $\lambda$, denoted by $\mu \subseteq \lambda$, if $\ell (\mu) \leq \ell(\lambda)$ and $\mu _i \leq \lambda _i$ for all $1 \leq i \leq \ell(\mu)$. The \emph{skew diagram} $\lambda / \mu$ of \emph{size} $(n-m) =| \lambda / \mu |$ is the array of boxes contained in $\lambda$ but not in $\mu$ when the array of boxes of $\mu$ is positioned in the top-left corner of the array of boxes of $\lambda$. A {\em horizontal strip} is a skew diagram with at most one square in each column. For a horizontal strip $\theta=\lambda/\mu$, we define
  \begin{align}\label{E:psi1}
  	\psi_{\lambda/\mu}(t)=\prod_{j\in J}(1-t^{m_j(\mu)}),
  \end{align}
  where $J$ is the set of integers $j\ge 1$ such that $\theta_j'<\theta_{j+1}'$. Since $\theta$ is a horizontal strip, the condition $\theta_j'<\theta_{j+1}'$ occurs only when $\theta_j'=0$ and $\theta_{j+1}'=1$.
  
Note that a semistandard Young tableau $T\in \mathrm{SSYT}(\lambda)$ is determined by a sequence of partitions $(\lambda^0=\emptyset,\lambda^1,\ldots,\lambda^r=\lambda)$ such that $\lambda^i/\lambda^{i-1}$ is a horizontal strip and each cell in $\lambda^i/\lambda^{i-1}$ is filled with the $i$th smallest entry in $T$. Then the HL function
  \begin{align}\label{E:SSYT_P}
  	P_{\lambda}(t)&=\sum_{T\in\mathrm{SSYT}(\lambda)}\psi_T(t)x_{c(T_1)}x_{c(T_2)}\cdots x_{c(T_{|\lambda|})}, \mbox{ where }\,
  	\psi_{T}(t)=\prod_{1\le i\le r}\psi_{\lambda^i/\lambda^{i-1}}(t)
  \end{align}
 and the boxes of $T$ are labelled by $T_1,T_2,\ldots,T_{|\lambda|}$. Note that (\ref{E:SSYT_P}) is equivalent to 
  \begin{align}\label{E:clam1}
  	c_{\lambda\mu}(t)=[m_{\mu}]P_{\lambda}=\sum_{T\in \mathrm{SSYT}(\lambda,\mu)}\psi_{T}(t)
  \end{align}
where $[b]f$ denotes the coefficient of $b$ in the expansion of $f$ with respect to the basis $\{b\}$ or the monomial $b$.
  
  \begin{example}
  	For $\lambda=(2,1)$, we derive a combinatorial formula for $P_{21}(t)$ by (\ref{E:SSYT_P}). The semistandard Young tableau 
  	\begin{align*}
  		T=\begin{ytableau}
  			$1$ &  $3$ \\
  			$2$ 
  		\end{ytableau}\,\,
  	\end{align*}
    is determined by the sequence $(\lambda^0=\emptyset,\lambda^1,\lambda^2,\lambda)$ where
  	$\lambda^1=(1)$ and $\lambda^2=(1,1)$. Therefore  $\psi_{T}(t)=\psi_{\lambda^1}(t)\psi_{\lambda^2/\lambda^1}(t)\psi_{\lambda/\lambda^2}(t)=1\cdot 1\cdot (1-t^2)=1-t^2$.
  	Similarly, the corresponding $\psi_{T}(t)$ for the semistandard  Young tableaux 
  	\begin{align*}
  		\begin{ytableau}
  			$1$ &  $1$ \\
  			$2$ 
  		\end{ytableau}\,\,,\,\, 
  		\begin{ytableau}
  			$1$ &  $2$ \\
  			$2$ 
  		\end{ytableau}\,\,,\,\,
  		\begin{ytableau}
  			$1$ &  $2$ \\
  			$3$ 
  		\end{ytableau}
  	\end{align*}
  	are $\psi_{T}(t)=1$, $\psi_{T}(t)=1$ and $\psi_{T}(t)=1-t$ from left to right. By (\ref{E:SSYT_P}),
  	$P_{21}(t)=m_{21}+(2-t-t^2)m_{111}$.
  \end{example} 
   Returning to Schur functions, Lascoux and Sch{\"u}tzenberger \cite{LS} proved that
   \begin{align}\label{E:sP}
   	s_{\lambda}=\sum_{\mu\le \lambda}K_{\lambda\mu}(t)P_{\mu}(t)
   \end{align}
   where $K_{\lambda\mu}(t)$ is called {\em Kostka polynomial} defined as the generating function of SSYTs counted by the $\mathrm{charge}$ statistic; see \cite[Chapter III.6, Equation (6.5) (i)]{Mac95} . In particular $K_{\lambda\mu}(1)=K_{\lambda\mu}$. 
   \begin{example}
   	$P_{21}(t)=m_{21}+(2-t-t^2)m_{111}$ and $s_{21}=P_{21}(t)+(t+t^2)P_{111}(t)$. 
   \end{example}

  \subsection{Symmetric functions in noncommuting variables} 
  We now introduce the graded Hopf algebra of symmetric functions in noncommuting variables, NCSym,
	\begin{align*}
		\mathrm{NCSym}=\mathrm{NCSym}^0 \oplus \mathrm{NCSym}^1 \oplus \cdots \subseteq \mathbb{Q}[[\bx]],
	\end{align*}
   where $\mathbb{Q}[[\bx]]=\mathbb{Q}[[\bx_1,\bx_2,\ldots]]$ is the algebra of formal power series in noncommuting variables $\bx=(\bx_1,\bx_2,\ldots)$, $\mathrm{NCSym}^0=\mathrm{span}\{1\}$, and the $n$th graded piece for $n\ge 1$ has the following bases given by Rosas and Sagan \cite{RS} and Aliniaeifard, Li and van Willigenburg \cite{alw:22}.
     \begin{alignat*}{3}
   	\mathrm{NCSym}^n&=\mathrm{span}\{\bm_{\pi}:\pi\vdash [n]\}&&=\mathrm{span}\{\bp_{\pi}:\pi\vdash [n]\}\\
   	&=\mathrm{span}\{\be_{\pi}:\pi\vdash [n]\}&&=\mathrm{span}\{\bh_{\pi}:\pi\vdash [n]\} =\mathrm{span}\{s_{\pi}:\pi\vdash [n]\}. &&
   \end{alignat*}
   For a set partition $\pi=\pi_1/ \cdots /\pi_{\ell(\lambda(\pi))}\vdash [n]$, we write $a\sim_{\pi} b$ if $a$ and $b$ are in the same block of $\pi$. The \emph{monomial symmetric function} $\bm_\pi$ in NCSym is defined by 
   $$ \bm_\pi=\sum_{(i_1, \dots, i_k)} \bx_{i_1}\cdots \bx_{i_{n}}\,$$
   where the sum is taken over all $k$-tuples $(i_1, \dots, i_n)$ such that $i_a=i_b$ if and only if $a \sim _{\pi} b$. The \emph{elementary symmetric function} $\be_\pi$ in NCSym is 
   $$ \be_\pi=\sum_{(i_1, \dots, i_k)} \bx_{i_1}\cdots \bx_{i_{n}}$$
   that is summed over all $k$-tuples $(i_1, \dots, i_k)$ with $i_a\ne i_b$ if $a\sim_{\pi}b$. The \emph{power-sum symmetric function} $\bp_{\pi}$ in NCSym is given by
   $$ \bp_\pi=\sum_{(i_1, \dots, i_k)} \bx_{i_1}\cdots \bx_{i_{k}}\,$$
   where the sum is ranged over all $k$-tuples $(i_1, \dots, i_k)$ with $i_a=i_b$ if $a\sim_{\pi} b$. The \emph{complete homogeneous  symmetric function} $\bh_\pi$ in NCSym is defined by 
   \begin{align}\label{E:htom}
   	 \bh_\pi=\sum_{\sigma\vdash [n]} (\sigma\wedge\pi)!\, \bm_\sigma.
   	\end{align} 
   Let $\mathfrak{S}_n$ be the set of permutations of $[n]$. Given $\delta\in \mathfrak{S}_n$ and a monomial  of degree $n$ in noncommuting variables, we define
   \begin{align}\label{E:delta}
   	\delta\circ(\bx_{i_1}\bx_{i_2}\cdots\bx_{i_n})= \bx_{i_{\delta^{-1}(1)}}\bx_{i_{\delta^{-1}(2)}}\cdots\bx_{i_{\delta^{-1}(n)}}
   \end{align}
   and extend it linearly \cite[Equation (5)]{RS}. This also yields that 
   \begin{align}\label{E:mbph}
   	\delta\circ \psi_{\pi}=\psi_{\delta\pi}
   \end{align}
   for $\psi\in \{\bm, \bp, \be,\bh\}$; see \cite[Section 2]{RS} and \cite[Equation (2.4)]{alw:22}. 
   
   Before we introduce the Schur function in NCSym, we need to present some properties of $\bh_{\pi}$ and noncommutative analogue of Leibniz' determinantal formula.
   Much like the $h$-basis in Sym (compare with (\ref{E:prod})), the $\bh$-functions multiply together as follows:
   \begin{lemma}\label{L:hmul}\cite[Corollary 2.41]{ber:10}\
   	For set partitions $\pi$ and $\sigma$ we have that
   	$$\bh_\pi \bh_\sigma = \bh_{\pi \mid \sigma}.$$
   \end{lemma}
   Hence, given a composition $\alpha=(\alpha_1,\ldots,\alpha_{\ell(\alpha)})$ we have $\bh_{[\alpha]}=\bh_{[\alpha_1]}\cdots \bh_{[\alpha_{\ell(\alpha)}]}$. Moreover, the noncommutative analogue of Leibniz' determinantal formula for any matrix $A=(a_{ij}) _{1\leq i,j\leq n}$ with noncommuting entries $a_{ij}$ is defined to be
   \begin{equation*}
   	{\bf det} (A) = \sum _{\varepsilon \in \mathfrak{S}_n} \mathrm{sgn} (\varepsilon) a_{1\varepsilon (1)}\cdots a_{n\varepsilon (n)}
   \end{equation*}
   that takes the product of the entries from the top row to the bottom row, and $\mathrm{sgn} (\varepsilon)$ is the sign of the permutation $\varepsilon$. Finally, the first Schur function $s_{\pi}$ in NCSym due to Aliniaeifard, Li and van Willigenburg \cite{alw:22} is defined to be
  \begin{align*}
  	s_{\pi}=\delta_{\pi}\circ {\bf det}\left(\frac{1}{(\lambda_i-\mu_j-i+j)!}\bh_{[\lambda_i-\mu_j-i+j]}\right)_{1\leq i,j\leq \ell(\lambda)}
  \end{align*}
  where we set $h_{[0]}= h_\emptyset = 1$ and  $h_{[i]}=0$ for $i<0$. 
  \begin{example}\label{Example:1schur}
  	The Schur function $s_{12/3}$ in $\mathrm{NCSym}$ is 
  	\begin{align*}
  		s_{12/3} &  = {{\bf det} \begin{pmatrix} \frac{1}{2!} \bh_{[2]}& \frac{1}{3!} \bh_{[3]}\\
  			\frac{1}{0!} \bh_{[0]}& \frac{1}{1!} \bh_{[1]}
  	\end{pmatrix}} = {\bf det} \begin{pmatrix} \frac{1}{2!} \bh_{12}& \frac{1}{3!} \bh_{123}\\
  		\frac{1}{0!} \bh_\emptyset & \frac{1}{1!} \bh_{1}
  	\end{pmatrix}\\
  	&= \frac{1}{2!} \bh_{12} \frac{1}{1!} \bh_{1} - \frac{1}{3!} \bh_{123}\frac{1}{0!} \bh_\emptyset = {\frac{1}{2} \bh_{12 \slashp 1} - \frac{1}{6} \bh_{123}} = \frac{1}{2} \bh_{12/3} - \frac{1}{6} \bh_{123}.
  \end{align*}
 By (\ref{E:htom}), we obtain 
 \begin{align*}
 	s_{12/3}=\frac{2}{3}\bm_{12/3}+\frac{1}{6}\bm_{13/2}+\frac{1}{6}\bm_{23/1}+\frac{1}{3}\bm_{1/2/3}.
 \end{align*} 
  \end{example}
    Note that (\ref{E:mbph}) does not hold for the Schur basis; namely, $\delta\circ s_{\pi}\ne s_{\delta\pi}$, as proved in \cite[Corollary 4.8]{alw:22}. 
   \subsection{Lifted and non-lifted HL functions}\label{ss:lift}
   The {\em projection map} introduced by Rosas and Sagan \cite[Section 2]{RS} is a linear map $\rho: \mathbb{Q}[[\bx_1, \bx_2, \dots]]\rightarrow \mathbb{Q}[[x_1, x_2,\dots]]$ that lets the variables commute. The image of NCSym bases under the projection map are given as below. 
   \begin{lemma}\cite[Theorem 2.1]{RS}\cite[Lemma 4.4]{alw:22}\label{lem:rho} 
   	Let $\pi$ be a partition. Then,
   	\begin{alignat*}{3}
   		\rho(\bm_\pi)&=\pi^! m_{\lambda(\pi)}, && \quad \rho(\bh_\pi)=\pi! h_{\lambda(\pi)},\\
   		\rho(\be_\pi)&=\pi! e_{\lambda(\pi)}, && \quad \rho(\bp_\pi)=p_{\lambda(\pi)},\\
   		\rho(s_{\pi})&=s_{\lambda(\pi)}.
   	\end{alignat*} 
   \end{lemma}
   Meanwhile, Rosas and Sagan defined the {\em lifting map} $\tilde{\rho}: \mathrm{Sym}\rightarrow \mathrm{NCSym}$ by 
   \begin{align}\label{E:liftm}
   	\tilde{\rho}(m_\lambda)=\frac{\lambda !}{|\lambda|!}\sum_{\lambda(\pi)=\lambda} \bm_\pi
   \end{align}
   and extended it linearly \cite[Equation (9)]{RS}. The lifting map $\tilde{\rho}$ is actually a right inverse of $\rho$; namely, $\rho\tilde{\rho}=\varepsilon$ is the identity map in Sym. \cite[Proposition 4.1]{RS}. 
   
   We introduce the {\em lifted HL functions} $\bP_{\lambda}(t)=\bP_{\lambda}(\bx;t)$ indexed by partitions that are obtained by applying the lifting map to the scaled HL functions $|\lambda|!P_{\lambda}(t)$. 
   \begin{align*}
   	\bP_{\lambda}(t)=|\lambda|!\,\tilde{\rho}(P_{\lambda}(t))
   	=\sum_{\mu\le \lambda}\mu!\,c_{\lambda\mu}(t) 
   	\sum_{\lambda(\sigma)=\mu}\bm_{\sigma}.
   \end{align*}
The {\em lifted Schur function} is defined by
\begin{align}\label{E:liftS}
	\bS_{\lambda}=\bP_{\lambda}(0)=|\lambda|!\,\tilde{\rho}(s_{\lambda})
	=\sum_{\mu} \mu! K_{\lambda\mu}\sum_{\lambda(\sigma)=\mu} \bm_\sigma,
\end{align}
which are introduced by Rosas and Sagan \cite{RS}. However, these polynomials are indexed by partitions, rather than set partitions; hence they do not form a basis for NCSym.

  Finally, we introduce the {\em non-lifted HL functions} $\bPP_{\lambda}(t)$ for partitions $\lambda$ in NCSym, beginning by specifying a reading order. The {\em reading order} is the total ordering on the boxes of the Young diagram given by reading them row by row, from top to bottom, and from left to right within each row. Then 
   \begin{align}\label{E:tabP2}
   \bPP_{\lambda}(t)=\bPP_{\lambda}(\bx;t)=\sum_{T\in \mathrm{SSYT}(\lambda)}\psi_{T}(t)
   	\bx_{c(T_{1})}
   	\bx_{c(T_{2})}\cdots \bx_{c(T_{n})},
   \end{align} 
   where the boxes of $T$ are labelled $T_1,T_2,\ldots,T_n$ by the reading order. Define $\bs_{\lambda}=\bPP_{\lambda}(0)$ to be the {\em non-lifted Schur function} in NCSym.    
 
   \section{Hall-Littlewood functions in NCSym}\label{S:3}
This section is devoted to defining the HL functions $\bP_{\pi}(t)=\bP_\pi(\bx;t)$ in noncommuting variables $\bx=(\bx_1,\bx_2,\ldots)$ with coefficients in the ring $\mathbb{Q}[t]$ for set partitions $\pi$, and to proving some nice properties of these polynomials. 

 \begin{definition}[HL functions in NCSym]\label{def:Ppi}		
	Let $\pi\vdash [n]$ be a set partition with $\lambda(\pi)=\lambda$. The HL function in noncommuting variables $\bx=(\bx_1,\bx_2,\ldots)$ is defined to be
	\begin{align}\label{E:pi4}
		\bP_{\pi}(t)=\bP_{\pi}(\bx,t)=\sum_{\mu\leq \lambda}\frac{c_{\lambda\mu}(t)}{\mu^!}
		\sum_{g\in G_{\pi}}g\circ \bm_{\delta_\pi [\mu]},
	\end{align}
	where $c_{\lambda\mu}(t)=[m_{\mu}]P_{\lambda}(t)$ and $G_\pi=\{\delta\in \mathfrak{S}_n\,\vert\, \delta\pi=\pi \}$ is the stabilizer group of $\pi$. The Schur function and Schur $P$-function  in noncommuting variables are defined by $\bs_{\pi}=\bP_{\pi}(0)$ and $\bP_{\pi}(-1)$, respectively.
\end{definition}
\begin{example}
	Let $\pi=12/3$, then $G_{\pi}=\mathfrak{S}_2$, 
	\begin{align*}
		\bP_{12/3}(t)=2\bm_{12/3}+\frac{2-t-t^2}{3}\bm_{1/2/3}\quad \mbox{ and }\quad\bs_{12/3}=2\bm_{12/3}+\frac{2}{3}\bm_{1/2/3}.
	\end{align*}
  Note that $\bs_{12/3}\ne s_{12/3}$ in Example \ref{Example:1schur}.
\end{example}
Our first main result shows that these polynomials form a $\mathbb{Q}[t]$-basis of NCSym.
 \begin{theorem}\label{T:basis} The set $\mathcal{B}_n=\{{\bP}_{\pi}(\bx;t): \pi \vDash [n]\}$ is a $\mathbb{Q}[t]$-basis of symmetric functions of homogeneous degree $n$ in noncommuting variables $\bx=(\bx_1,\bx_2,\ldots)$, and it is invariant under any permutation of $[n]$. That is, $w\mathcal{B}_n=\mathcal{B}_n$ for any $w\in  \mathfrak{S}_n$.
\end{theorem}
We usually omit the variables $\bx$ and simply write $\bP_{\pi}(t)$. Before proving Theorem \ref{T:basis}, we discuss the stabilizer group $G_{\pi}$ and establish two auxiliary lemmas. 

For any subset $B\subseteq [n]$, let $\mathfrak{S}_{B}=\{\delta\in \mathfrak{S}_n \,\vert\, \delta B=B\}$ be the stabilizer subgroup of $B$ in $\mathfrak{S}_n$.
Given a set partition $\pi=\pi_1/\cdots/\pi_{\ell(\pi)}\vdash [n]$, let $\lambda=\lambda(\pi)=\langle 1^{m_1}, 2^{m_2},\dots, k^{m_k}\rangle$. Then for each $m_j>1$, let $\pi_{j_1},\dots, \pi_{j_{m_j}}$ be the blocks in $\pi$ of size $j$ where $\pi_{j_a}=\{x_{a,1}<\cdots<x_{a, j}\}$ for $1\le a\le m_j$.  The subgroup $H_j$ is defined by 
\begin{align}\label{Def:H}
H_j=\langle s^j_i\,|\, 1\le i<m_j\rangle
\end{align}
of $\mathfrak{S}_n$, where $s^j_i=(x_{i,1}\, x_{i+1,1})(x_{i,2}\,x_{i+1,2})\cdots (x_{i,j}\,x_{i+1,j})$ is a product of disjoint transpositions. Then we have 
\begin{align}\label{eq:G_pi_1}
	G_{\pi}&=R_{\pi}\rtimes H_{\pi},\quad \mbox{ where }\,R_{\pi}=\prod_{i=1}^{\ell(\pi)} \mathfrak{S}_{\pi_i}\,\mbox{ and }\,H_{\pi}=\prod_{m_j>1} H_j.
\end{align}
Since the group $H_j$ is isomorphic to $\mathfrak{S}_{m_j}$, the order of $G_\pi$ is equal to $\lambda! \lambda^{!}$. Moreover, $R_\pi$ stabilizes each block of $\pi$ and $H_\pi$ permutes the blocks of equal size, and
 \begin{equation}\label{eq:G_pi}
	G_\pi=\bigsqcup_{h\in H_\pi} hR_{\pi}=\bigsqcup_{h\in H_\pi} R_{\pi}h\,.
\end{equation}
 \begin{example}
	If  $\pi=1\,2\,3\,4/5\,6\,7\,8/ 9\,10\,11\,12$, then $\lambda(\pi)=\langle 4^3\rangle$ and by (\ref{Def:H}),
	\begin{align*}
	H_4=\langle(1\,5)(2\,6)(3\,7)(4\,8), (5\, 9)(6\,10)(7\,11)(8\,12)\rangle.
    \end{align*}
	Hence, by (\ref{eq:G_pi_1}), $G_{\pi}=\left(\mathfrak{S}_{\{1, 2,3,4\}}\times\mathfrak{S}_{\{ 5, 6, 7, 8\}}\times\mathfrak{S}_{\{9, 10, 11, 12\}}\right)\rtimes  H_4$.
\end{example}	
 \begin{lemma}\label{lem:R} Let $\pi$ be a set partition of $[n]$ and $w$ be a permutation of $[n]$. Then, 
	$$R_{w\pi}=wR_{\pi}w^{-1}\,.$$
\end{lemma}
\begin{proof}
By definition,  $u\in R_{w\pi}$ if and only if $u(w \pi_i)=w \pi_i$ for all $1\le i\le \ell(\pi)$, 
which is equivalent to $(w^{-1}uw) \pi_i=\pi_i$. Since  $(w^{-1}uw) \pi_i=\pi_i$ for all $i$ if and only if $w^{-1}uw\in R_\pi$, the proof is complete.
 \end{proof}
\begin{lemma}\label{lem:pi_to_sigma} Let $\pi$ and $\sigma$ be set partitions of $[n]$ with $\lambda(\pi)=\lambda(\sigma)$. If $w\pi=\sigma$ for some permutation $w$ of $[n]$. Then the following statements hold:
	\begin{enumerate}
		\item $H_{\sigma}=wH_\pi w^{-1}$.
		\item $\bP_{\sigma}(t)=w\circ\bP_{\pi}(t)$.
	\end{enumerate}
	In particular, $\bP_{\sigma}(t)=\delta_\sigma(\delta_\pi)^{-1}\bP_{\pi}(t)$.
\end{lemma}
\begin{remark}
It follows from Lemma \ref{lem:pi_to_sigma} (2) that
$\bs_{\pi}\ne s_{\pi}$, since $\bs_{\sigma}=w\bs_{\pi}$,  whereas $s_{\sigma}\ne w s_{\pi}$ as proved in \cite[Corollary 4.8]{alw:22} for all permutations $w$ of $[n]$. 
\end{remark}
\begin{proof}
	We begin by proving (1). 
	For each generator $s^j_i=(x_{i,1}\, x_{i+1,1})(x_{i,2}\,x_{i+1,2})\cdots (x_{i,j}\,x_{i+1,j})$ of $H_j\leq H_\pi$, 
	note that $ws^j_iw^{-1}=(w(x_{i,1})\, w(x_{i+1,1}))(w(x_{i,2})\,w(x_{i+1,2}))\cdots (w(x_{i,j})\,w(x_{i+1,j}))$ is a generator of $H_{\sigma}$. This yields $H_{\sigma}=wH_\pi w^{-1}$.
	
	By (1), Lemma \ref{lem:R} and (\ref{eq:G_pi}), we find that 
	\begin{align}\label{E:wG}
		wG_\pi &= w\bigsqcup_{h\in H_\pi} hR_{\pi}=w\bigsqcup_{h'\in H_{\sigma}} w^{-1}h'wR_{\pi}
		=\bigsqcup_{h'\in H_{\sigma}}h' R_{\sigma}w = G_{\sigma}w\,,
	\end{align}
	thus $G_\sigma= wG_{\pi}w^{-1}$. In particular, 
	$\delta_\pi G_{[\lambda]}=G_\pi\delta_\pi$. Since $\delta_\sigma [\lambda]=\sigma=w\pi=w\delta_\pi [\lambda]$, it follows that $\delta_\sigma=(w\delta_\pi)g^*=wg^\dagger\delta_\pi$ for some $g^*\in G_{[\lambda]}$ and $g^\dagger\in G_{\pi}$. Therefore, 
	\begin{align*}
		\bP_{\sigma}(t)&=\sum_{\mu\leq \lambda}\frac{c_{\lambda\mu}(t)}{\mu^!}
		\sum_{g\in G_{\sigma}}g\circ \bm_{\delta_\sigma [\mu]}	=\sum_{\mu\leq \lambda}\frac{c_{\lambda\mu}(t)}{\mu^!}
		\sum_{g\in wG_{\pi}w^{-1}}g\circ \bm_{wg^\dagger\delta_\pi [\mu]}	\\
		&= \sum_{\mu\leq \lambda}\frac{c_{\lambda\mu}(t)}{\mu^!}
		\sum_{g'\in G_{\pi}} (wg'g^\dagger)\circ \bm_{\delta_\pi [\mu]}= w\sum_{\mu\leq \lambda}\frac{c_{\lambda\mu}(t)}{\mu^!}
		\sum_{g''\in G_{\pi}} g''\circ \bm_{\delta_\pi [\mu]}\\
		&=w\circ\bP_{\pi}(t),
	\end{align*}
where the relation $\delta\circ \bm_{\pi}=\bm_{\delta\pi}$ is applied, as stated in (\ref{E:mbph}). 
Finally, since $\pi=\delta_\pi [\lambda]$, we write $[\lambda]=(\delta_{\pi})^{-1}\pi$ and consequently $\sigma=\delta_\sigma [\lambda]=\delta_\sigma (\delta_\pi)^{-1} \pi$. Take $w=\delta_\sigma (\delta_\pi)^{-1}$; then by (2), $\bP_{\sigma}(t)=\delta_\sigma(\delta_\pi)^{-1}\bP_{\pi}(t)$. 
\end{proof}
Now we are ready to prove Theorem \ref{T:basis}.

{\em Proof of Theorem \ref{T:basis}}. By Definition \ref{def:Ppi}, the only set partition $\sigma$ such that $[\bm_{\sigma}]\bP_{\pi}(t)\ne 0$ and $\lambda(\sigma)=\lambda(\pi)$ is $\pi$ itself. Since $c_{\lambda\lambda}(t)=[m_{\lambda}]P_{\lambda}=1$ and $|G_{\pi}|=\lambda!\lambda^!$, we have
\begin{align*}
	[\bm_{\pi}]\bP_{\pi}(t)=[\bm_{\pi}]\frac{c_{\lambda\lambda}(t)}{\lambda^!}\sum_{g\in G_{\pi}} g\circ \bm_{\delta_{\pi}[\lambda]}=\frac{c_{\lambda\lambda}(t)}{\lambda^!}|G_{\pi}|=\lambda!.
\end{align*} 
Moreover, $[\bm_{\sigma}]\bP_{\pi}(t)=0$ for set partitions $\sigma$ with $\lambda(\sigma)\not\le \lambda(\pi)$. This implies that the transition matrix from $\bP_{\pi}$ to $\bm_{\sigma}$ is a triangular matrix with nonzero entries on the main diagonal. Since the set $\{\bm_\sigma:\sigma\vdash [n]\}$ is a $\mathbb{Q}$-basis of NCSym, and therefore also a $\mathbb{Q}[t]$-basis of NCSym, 
it follows that $\mathcal{B}_n$ is also a $\mathbb{Q}[t]$-basis of NCSym. Note that the symmetric group $\mathfrak{S}_n$ acts transitively on $\mathcal{B}_n$ by Lemma~\ref{lem:pi_to_sigma} (2), and hence $w\mathcal{B}_n=\mathcal{B}_n$ for any $w\in \mathfrak{S}_n$.
\qed

Our second main result proves that the HL functions $P_{\pi}(t)$ refine the lifted HL function $\bP_{\lambda}(t)$.  
 \begin{theorem}\label{T:sumP1} 
 	For any partition $\lambda$, the lifted HL function is given by
 	\begin{align}\label{E:refineP}
 		\bP_{\lambda}(t)=\sum_{\lambda(\pi)=\lambda}\bP_{\pi}(t).
 	\end{align}
 	Consequently, when $t=0$, we have
 	\begin{align}\label{E:refineS}
 		\bS_{\lambda}=\sum_{\lambda(\pi)=\lambda}\bs_{\pi}
 		=\sum_{\mu\le\lambda}\sum_{\lambda(\pi)=\mu}K_{\lambda\mu}(t)\bP_{\pi}(t).
 	\end{align}
 Moreover, $\bS_{n}=\bh_{[n]}$ and $\bS_{1^n}=\be_{[n]}$. If $\lambda$ is a partition with distinct parts, then 
 	\begin{align}\label{E:refineQ}
 		\bP_{\lambda}(-1)=\sum_{\lambda(\pi)=\lambda}\bP_{\pi}(-1)&= \sum_{\mu}g_{\lambda\mu}\bS_{\mu},
 	\end{align}
 	where $g_{\lambda\mu}=[s_{\mu}]P_{\lambda}(-1)$ is the number of restricted marked tableaux of shape $\mu$ and weight $\lambda$ proved by Stembridge \cite[Page 131]{Stem:89}; see also \cite[Chapter III, (8.17)--(8.18)]{Mac95}. The polynomial $P_{\lambda}(-1)$ is called the Schur $P$-function. 
 \end{theorem}
  A polynomial in NCSym is said to be {\em $\HL$-positive} (resp. $s$-positive), if it can be written as a linear combination of HL functions $\bP_{\pi}(t)$ (resp. $\bs_{\pi}$) in NCSym with coefficients in $\mathbb{N}[t]$ (resp. $\mathbb{N}$). Theorem \ref{T:sumP1} implies that the lifted HL function $\bP_{\lambda}(t)$ is $\HL$-positive. In particular, $\bS_{\lambda}$ and $\bP_{\lambda}(-1)$ are $s$-positive. We now prove Theorem \ref{T:sumP1}.
   \begin{proof}
   We note that, for a partition $\lambda \vdash n$, the set  $\{\pi\vdash [n] : \lambda(\pi)=\lambda\}$ is the orbit of $[\lambda]$ under the action of $\mathfrak{S}_n$, and the stabilizer group of $[\lambda]$ is $G_{[\lambda]}$. Therefore, by Lemma~\ref{lem:pi_to_sigma} (2), 
    \begin{align*} 
    	\sum_{\lambda(\pi)=\lambda} \bP_{\pi}(t) &= \sum_{w\in \mathfrak{S}_n/G_{[\lambda]}} w\circ\bP_{[\lambda]}(t)
    	= \sum_{w\in \mathfrak{S}_n/G_{[\lambda]}} \sum_{\mu\leq \lambda}\frac{c_{\lambda\mu}(t)}{\mu^!}\sum_{g\in G_{[\lambda]}}wg\circ \bm_{[\mu]}\\
    	&= \sum_{\mu\leq \lambda}\frac{c_{\lambda\mu}(t)}{\mu^!}\sum_{g\in \mathfrak{S}_n }g\circ \bm_{[\mu]}
    	= \sum_{\mu\leq \lambda}\frac{c_{\lambda\mu}(t)}{\mu^!} (\mu !\mu^!)\sum_{\lambda(\sigma)=\mu} \bm_\sigma\\ &=\sum_{\mu\leq\lambda} \mu! c_{\lambda\mu}(t)\sum_{\lambda(\sigma)=\mu} \bm_\sigma=\bP_{\lambda}(t)\,.
    \end{align*}
    Consequently, (\ref{E:refineS}) is established by setting $t=0$, where $c_{\lambda\mu}(0)=K_{\lambda\mu}$, and by applying the lifting map $\tilde{\rho}$ to both sides of (\ref{E:sP}). Similarly, (\ref{E:refineQ}) is obtained by setting $t=-1$, and then evaluating both sides of $P_{\lambda}(-1)=\sum_{\mu}g_{\lambda\mu}s_{\mu}$ (for any partition $\lambda$ with distinct parts) under the lifting map. 
    
    It remains to prove $\bS_{n}=\bh_{[n]}$ and $\bS_{1^n}=\be_{[n]}$. Since the Kostka number $K_{(n)\mu}=1$ for all partitions $\mu\vdash n$,  by definitions (\ref{E:liftS}) and (\ref{E:htom}) we have 
    \begin{align*}
    	\bS_{n}=\sum_{\mu\,\vdash n} \mu!\sum_{\lambda(\sigma)=\mu} \bm_\sigma
    	=\sum_{\sigma\,\vdash [n]}\sigma! \bm_{\sigma}=\bh_{[n]}.
    \end{align*}
    Finally, 
    $\bS_{1^n}=\bm_{1/2/\cdots/n}=\be_{[n]}$ by definition, which completes the proof.
     \end{proof}
    \begin{example}\label{E:22} 
        We give an example of (\ref{E:refineP}). 	
    	Let $\lambda=(2,2)$, then $[\lambda]=1\,2/3\,4$ and 
    	\begin{align*}
    		G_{12/34}&=( \mathfrak{S}_{\{1, 2\}}\times \mathfrak{S}_{\{3,4\}} )\rtimes \langle (1\,3)(2\,4) \rangle\\
    		&=( \mathfrak{S}_{\{1, 2\}}\times \mathfrak{S}_{\{3,4\}} )\cup (1\,3)(2\,4)( \mathfrak{S}_{\{1, 2\}}\times \mathfrak{S}_{\{3,4\}} )\\
    		&= \{\varepsilon, (1\,2), (3\,4), (1\,2)(3\,4), (1\,3)(2\,4), (1\,4\,2\,3), (1\,3\,2\,4), (1\,4)(2\,3)\} \,.
    	\end{align*}
    	There are three set partitions $\pi\vdash [4]$ such that $\lambda(\pi)=(2,2)$, listed as below. 
    	$$[\lambda]=1\,2/3\,4, \quad  \pi=1\,3/2\,4 \,\,   \mbox{ and }\,\,\sigma= 1\,4/2\,3\,,$$
    	and $\{\delta_{[\lambda]}=id, \delta_\pi=(2\,3), \delta_\sigma=(2\, 4\, 3)\}$ is a set of coset representatives of $ \mathfrak{S}_4/G_{[\lambda]}$. 
    	One can  also check that 
    	\begin{align*}
    		G_\pi&=( \mathfrak{S}_{\{1, 3\}}\times \mathfrak{S}_{\{2,4\}} )\rtimes \langle (1\,2)(3\,4) \rangle=\delta_\pi G_{[\lambda]}(\delta_\pi)^{-1} \,, \\
    		G_\sigma&=( \mathfrak{S}_{\{1, 4\}}\times \mathfrak{S}_{\{2,3\}} )\rtimes \langle (1\,2)(4\,3) \rangle=\delta_\sigma G_{[\lambda]}(\delta_\sigma)^{-1}\,,
    	\end{align*}
       as claimed in (\ref{E:wG}). By (\ref{E:pi4}), the corresponding HL functions are given by
    	\begin{align*}
    		\bP_{12/34}(t)&={1\over 2!} (8 \bm_{12/34})+{{1-t}\over 2!}( 4\bm_{12/3/4} +4 \bm_{34/1/2})\\
    		&\qquad + {{(t+2)(t-1)^2}\over 4!}(8 \bm_{1/2/3/4}),
    	\end{align*}
    	$\bP_{13/24}(t)=(2\,3)\bP_{12/34}(t)$ and $\bP_{14/23}(t)=(2\,4\,3)\bP_{12/34}(t)$. As a result, 
    	\begin{align*}
    		&\,\,\quad\bP_{12/34}(t)+\bP_{13/24}(t)+ \bP_{14/23}(t)\\
    		&=  4( \bm_{12/34}+\bm_{13/24} +\bm_{14/23})\\
    		&\qquad +\frac{1-t}{2}(\bm_{12/3/4}+\bm_{34/1/2}+\bm_{13/2/4}+\bm_{24/1/3}+\bm_{14/2/3}+\bm_{23/1/4})\\
    		&\qquad+\frac{(t+2)(t-1)^2}{3}\bm_{1/2/3/4}\\
    		&= \bP_{22}(t)\,.
    	\end{align*}
    \end{example}    
    Just as the HL function $P_{\lambda}(t)$ in Sym interpolates between the Schur function $s_{\lambda}=P_{\lambda}(0)$ and the monomial symmetric function $m_{\lambda}=P_{\lambda}(1)$, 
    the HL function $\bP_{\pi}(t)$ in NCSym also exhibits such nice property. 
    \begin{proposition}\label{prop:basic} Let $\pi$ be a set partition of $[n]$ and $\lambda=\lambda(\pi)$. Then, 
    	\begin{enumerate}
    		\item $\bP_{\pi}(1)= \lambda!\, \bm_\pi$.
    		\item $\rho(\bP_{\pi}(t))=\lambda ! \lambda^! P_\lambda(t)$. Consequently, $\rho(\bs_{\pi})=\lambda!\lambda^!s_{\lambda}$.
    	\end{enumerate}
    \end{proposition}
    \begin{proof}
    	Since $c_{\lambda\mu}(1)=\delta_{\lambda\mu}$ and $|G_\pi|=\lambda !\lambda^!$, then (1) follows easily from the definition (\ref{E:pi4}) of $\bP_{\pi}(t)$. Furthermore, Lemma \ref{lem:rho} yields that  
    	\begin{align*}
    		\rho(\bP_{\pi}(t))=\sum_{\mu\leq \lambda}\frac{c_{\lambda\mu}(t)}{\mu^!}
    		\lambda ! \lambda^!(\mu^! m_{\mu})= \lambda ! \lambda^!\sum_{\mu\leq \lambda}c_{\lambda\mu}(t) m_\mu= \lambda ! \lambda^! P_\lambda(t),
    	\end{align*}
    	and thus $\rho(\bs_{\pi})=\lambda!\lambda^!s_{\lambda}$ by taking $t=0$.
    \end{proof}
     We now provide a noncommutative version of the classical HL expansion of Schur function as given in (\ref{E:sP}).
    \begin{proposition} 
    	Let $\pi$ be a set partition with $\lambda(\pi)=\lambda$, let
    	\begin{align*}
    		\widetilde{\bP}_{\pi}(t)&=\sum_{\mu\leq \lambda}\frac{c_{\lambda\mu}(t)}{\mu^!}\,\bm_{\delta_\pi [\mu]}
    	\end{align*}
        where $c_{\lambda\mu}(t)=[m_{\mu}]P_{\lambda}$. Then we have
    	\begin{align}\label{E:stoP}
    		\bs_{\pi}=\sum_{\mu\le \lambda}K_{\lambda\mu}(t)\sum_{g\in G_{\pi}}g\circ \widetilde{\bP}_{\delta_{\pi}[\mu]}(t).
    	\end{align}
    \end{proposition}
    \begin{proof}
    	This is proved by examining the transition matrices between the Schur basis, the HL basis and the monomial basis of Sym. Let $K(t)=(K_{\lambda\mu}(t))_{\lambda,\mu}$ and $c(t)=(c_{\lambda\mu}(t))_{\lambda,\mu}$ be two matrices. Then $K(1)=(K_{\lambda\mu})_{\lambda,\mu}$ is the matrix of Kostka numbers, and (\ref{E:sP}) is equivalent to $K(t)c(t)=K(1)$, as shown by the following argument. By (\ref{E:stom2}), (\ref{E:Pm11}) and (\ref{E:sP}), 
    	\begin{align*}
    		s_{\lambda}=\sum_{\mu\le \lambda} K_{\lambda\mu}(t) \sum_{\nu\le \mu}c_{\mu\nu}(t)m_{\nu}
    		&=\sum_{\nu\le \lambda}\sum_{\nu\le \mu\le \lambda}K_{\lambda\mu}(t)c_{\mu\nu}(t)m_{\nu}
    		=\sum_{\nu\le \lambda}K_{\lambda\nu} m_{\nu}.
    	\end{align*}
    	The last equality implies that $K(t)c(t)=K(1)$. Since $c_{\lambda\mu}(0)=K_{\lambda\mu}$, we are led to
    	\begin{align*}
    		\widetilde{\bP}_{\pi}(0)&=\sum_{\nu\le \lambda}\frac{K_{\lambda\nu}}{\nu^!}\,\bm_{\delta_{\pi}[\nu]}
    		=\sum_{\nu\le \lambda}\frac{1}{\nu^!}\sum_{\nu\le \mu\le\lambda}K_{\lambda\mu}(t)c_{\mu\nu}(t)\,\bm_{\delta_{\pi}[\nu]},\\
    		&=\sum_{\mu\le \lambda}K_{\lambda\mu}(t)\sum_{\nu\le \mu}\frac{c_{\mu\nu}(t)}{\nu^!}\,\bm_{\delta_{\pi}[\nu]},\\
    		&=\sum_{\mu\le \lambda}K_{\lambda\mu}(t)\widetilde{\bP}_{\delta_{\pi}[\mu]}(t).
    	\end{align*}
        Consequently, (\ref{E:stoP}) is obtained by noting that $\bs_{\pi}=\sum_{g\in G_{\pi}}g\circ \widetilde{\bP}_{\pi}(0)$.
    \end{proof}
   \begin{example}
    	Let $\pi=123/4=[3,1]$, then
    	\begin{align*}
    		\widetilde{\bP}_{123/4}(0)&=K_{31,31}(t)\widetilde{\bP}_{123/4}(t)+K_{31,22}(t)\widetilde{\bP}_{12/34}(t)\\
    		&\qquad+K_{31,211}(t)\widetilde{\bP}_{12/3/4}(t)+K_{31,1111}(t)\widetilde{\bP}_{1/2/3/4}(t).\\
    		&=\widetilde{\bP}_{123/4}(t)+t\widetilde{\bP}_{12/34}(t)+(t+t^2)\widetilde{\bP}_{12/3/4}(t)+(t^3+t^4+t^5)\widetilde{\bP}_{1/2/3/4}(t).
    	\end{align*}
    	  Since $G_{\pi}=\mathfrak{S}_{3}$, (\ref{E:stoP}) yields that 
    	\begin{align*}
    		\bs_{123/4}&=6\widetilde{\bP}_{123/4}(t)+2t(\widetilde{\bP}_{12/34}(t)+\widetilde{\bP}_{13/24}(t)+\widetilde{\bP}_{14/23}(t))\\
    		&\qquad+2(t+t^2)(\widetilde{\bP}_{12/3/4}(t)+\widetilde{\bP}_{13/2/4}(t)+\widetilde{\bP}_{23/1/4}(t))\\
    		&\qquad+6(t^3+t^4+t^5)\widetilde{\bP}_{1/2/3/4}(t).
    	\end{align*}
    \end{example}
   \section{Product rule for Hall-Littlewood functions in NCSym}\label{S:4}
   
   The purpose of this section is to provide the product rule for HL functions in NCSym. We develop a star-product rule in the context of NCSym to describe how the star product $\bP_{\mu}(\bx;t)\star\bPP_{\nu}(\bx;t)$ of a lifted and a non-lifted HL functions expands as a linear combination of HL functions $\bP_{\pi}(\bx;t)$ with coefficients in $\mathbb{Z}[t]$. Throughout this section, we omit the variable $t$ for brevity.
   
   Let us briefly explain how we arrive at the star product. Since the HL functions $\bP_{\tau}$ form a $\mathbb{Q}[t]$-basis of NCSym by Theorem \ref{T:basis}, the multiplication $\bP_{\pi}\bP_{\sigma}$ of two HL functions is a linear combination of $\bP_{\tau}$ with $\mathbb{Q}[t]$-coefficients. That is, there exist functions $d_{\pi\sigma}^{\tau}(t)\in\mathbb{Q}[t]$ such that 
   \begin{align}\label{E:prod1}
   	\bP_{\pi}\bP_{\sigma}=\sum_{\tau}d_{\pi\sigma}^{\tau}(t)\bP_{\tau}.
   \end{align}
   This is in analogy with the product rule for HL functions in Sym \cite[Chapter III, Section 3]{Mac95}: namely, there are polynomials $c_{\mu\nu}^{\lambda}(t)\in\mathbb{Z}[t]$ such that 
   \begin{align}\label{E:prodSym}
   	P_{\mu}P_{\nu}=\sum_{\lambda}c_{\mu\nu}^{\lambda}(t)P_{\lambda}.
   \end{align} 
   If $\nu=\langle1^r\rangle$ is a column, then $P_{\nu}=m_{\nu}=e_{r}$ by (\ref{E:Pm11}) and $c_{\mu\nu}^{\lambda}(t)\in\mathbb{N}[t]$; see \cite[Chapter III, Equation (3.2)]{Mac95}. However, this is not true for (\ref{E:prod1}); namely, if  $\sigma=1/2/\cdots/r$, we have $\bP_{\sigma}=\be_{[r]}=\bs_{\langle1^r\rangle}$, but $d_{\pi\sigma}^{\tau}(t)\not\in\mathbb{Q}^+(t)$. For example, take $\pi=12/3$ and $\sigma=1$, we have
   \begin{align}
   	\bP_{12/3}\bs_{1}
   	&=\frac{1}{3}\bP_{124/3}+\frac{t+5}{12}\bP_{12/34}-\frac{1-t}{12}\bP_{13/24}-\frac{1-t}{12}\bP_{14/23}\notag\\
   	&\quad+\frac{3+8t+t^2}{12}\bP_{12/3/4}+\frac{1-t^2}{12}(\bP_{14/2/3}+\bP_{24/1/3})\notag\\
   	&\quad+\frac{(1-t)^2}{12}(\bP_{13/2/4}+\bP_{23/1/4})-\frac{(1-t)^2}{12}\bP_{34/1/2}.\label{E:eg1}
   \end{align}
   We then consider the sum $\sum_{\lambda(\pi)=\mu}\bP_{\pi}\bs_1=\bP_{\mu}\bs_1$, and it turns out that $[\bP_{\tau}]\bP_{\mu}\bs_1\not \in \mathbb{Q}^+(t)$ for some set partitions $\tau$ and partitions $\mu$. We continue by illustrating this with our running example, namely the case $\mu=(2,1)$. It follows from (\ref{E:eg1}) that
   \begin{align*}
   	\bP_{21}\bs_1&=\bP_{12/3}\bs_1+\bP_{13/2}\bs_1+\bP_{23/1}\bs_1=\bP_{12/3}\bs_1+(23)(\bP_{12/3}\bs_1)+(13)(\bP_{12/3}\bs_1)\\
   	&=\frac{1}{3}(\bP_{124/3}+\bP_{234/1}+\bP_{134/2})+\frac{1+t}{4}(\bP_{12/34}+\bP_{13/24}+\bP_{14/23})\\
   	&\qquad+\frac{3t^2+4t+5}{12}(\bP_{12/3/4}+\bP_{13/2/4}+\bP_{23/1/4})\\
   	&\qquad+\frac{2t-3t^2+1}{12}(\bP_{14/2/3}+\bP_{24/1/3}+\bP_{34/1/2}).
   \end{align*}
   This motivates us to introduce the star multiplication $\star$ of HL functions and prove its positive HL  expansion for certain specializations. For instance, $[\bP_{\tau}](\bP_{\mu}\star \bs_1)=[P_{\lambda(\tau)}](P_{\mu}s_1)\in \mathbb{N}[t]$ for all partitions $\mu$ and set partitions $\tau$.
   
   For any integer $1\le r<n$, let $\S_{r,n}$ be a set of coset representatives of the subgroup $\mathfrak{S}_{n-r}$ in $\mathfrak{S}_n$, defined as follows, where we recall that $(i\,\,j)\in \mathfrak{S}_n$ is the transposition that swaps $i$ and $j$, 
   \begin{align*}
   	\S_{r,n}=\{(i_r\,n)(i_{r-1}\,n-1)\cdots(i_1\,n-r+1): 1\le i_k\le n-r+k\textrm{ and }1\le k\le r\}.
   \end{align*}
   Then $|\S_{r,n}|=n!/(n-r)!$ and 
   \begin{align}\label{E:coset}
   	\mathfrak{S}_n=\bigsqcup_{\delta\in \S_{r,n}} \delta\,\mathfrak{S}_{n-r}.
   \end{align} 
   Only when $i_k=n-r+k$ for all $1\le k\le r$, we get the identity permutation $\varepsilon$ in $\S_{r,n}$. Given two polynomials $\CMcal{F},\CMcal{G}\in \mathbb{Q}[[\bx]]$ of homogeneous degrees $n-r$ and $r$, respectively, we define their star product by
   \begin{align}\label{E:starproduct}
   	\CMcal{F}\star \CMcal{G}=\sum_{\delta\in\S_{r,n}}\delta^{-1}\circ (\CMcal{F}\CMcal{G})
   \end{align}
   where the action of $\mathfrak{S}_n$ on $\mathbb{Q}[[\bx]]$ is defined in (\ref{E:delta}).
   
   Our third main result provides the star-product rule for a lifted and a non-lifted HL functions in NCSym, as a noncommutative analogue to the multiplication rule (\ref{E:prodSym}) for two HL functions. 
   
   \begin{theorem}\label{T:pieri1}
   	For $n\ge 2$ and $r\in [n-1]$,
   	let $\mu\vdash (n-r)$ and $\nu\vdash r$ be two partitions. Then 
   	\begin{align*}
   		\bP_{\mu}\star \bPP_{\nu}=\sum_{\lambda}c_{\mu\nu}^{\lambda}(t)\,\bP_{\lambda}
   	\end{align*}
  summed over partitions $\lambda$ of $n$, and $c_{\mu\nu}^{\lambda}(t)=[P_{\lambda}](P_{\mu}P_{\nu})\in \mathbb{Z}[t]$. 
   Consequently, when $t=0$,
   \begin{align}\label{E:LRs}
   	\bS_{\mu}\star \bs_{\nu}=\sum_{\lambda}c_{\mu\nu}^{\lambda}\bS_{\lambda},
   \end{align}
   where $c_{\mu\nu}^{\lambda}=c_{\mu\nu}^{\lambda}(0)$ is the Littlewood-Richardson coefficient.  
   \end{theorem}
  We make a few remarks before proceeding to the proof of Theorem \ref{T:pieri1}.
  
  Since $\bP_{\sigma}$ is a refinement of $\bP_{\lambda(\sigma)}$ by Theorem \ref{T:sumP1} and $c_{\mu\nu}^{\lambda}(t)\in \mathbb{N}[t]$ for $\nu=\langle 1^r\rangle$, it follows from Theorem \ref{T:pieri1} that 
   \begin{align*}
   	\sum_{\lambda(\pi)=\lambda}[\bP_{\sigma}](\bP_{\pi}\star \be_{[r]})=[\bP_{\sigma}](\bP_{\lambda}\star \be_{[r]})=[P_{\lambda(\sigma)}](P_{\lambda}e_r)\in \mathbb{N}[t].
   \end{align*}
   Nevertheless, the coefficient $[\bP_{\sigma}](\bP_{\pi}\star \be_{[r]})\not \in \mathbb{N}[t]$, even for the simplest case where $r=1$. See Example \ref{Example:1}. It is also important to note that if we replace $\S_{r,n}$ with another set of coset representatives of $\mathfrak{S}_{n-r}$ in $\mathfrak{S}_n$ in the definition of the star product, then we obtain a different product for $\bP_{\pi}\star \bs_1$, whereas the product $\bP_{\lambda}\star \bs_1$ remains the same. This will be evident from the proof of Theorem \ref{T:pieri1}.
   
   Now we establish Theorem \ref{T:pieri1}.
   
   {\em Proof of Theorem \ref{T:pieri1}.} We first express the lifted HL function $\bP_{\lambda}(t)$ in terms of marked tableaux, following the ideas for lifted Schur functions in \cite[Section 6]{RS}. Let $\lambda$ be a partition of $n$, then we claim that 
   \begin{align}\label{E:tabP}
   	\bP_{\lambda}(\bx;t)=\sum_{w\in \mathfrak{S}_n}\sum_{T\in \mathrm{SSYT}(\lambda)}\psi_{T}(t)
   	\bx_{c(T_{w(1)})}
   	\bx_{c(T_{w(2)})}\cdots \bx_{c(T_{w(n)})},
   \end{align} 
   where the boxes of $T$ are labelled $T_1,T_2,\ldots,T_n$ in an arbitrary order, and $c(T_i)$ is the content of box $T_i$. 
   By (\ref{E:clam1}), we write 
   \begin{align}\label{E:tabPP}
   	P_{\lambda}(t)=\sum_{\mu\le \lambda}c_{\lambda\mu}(t)m_{\mu}=\sum_{\mu\le\lambda}\sum_{T\in \mathrm{SSYT}(\lambda,\mu)}\psi_{T}(t)m_{\mu},
   \end{align}
   which by (\ref{E:liftm}) implies that for any set partition $\sigma\vdash [n]$ with $\lambda(\sigma)=\mu$, 
   \begin{align}\label{E:P2}
   	[\bm_{\sigma}]\bP_{\lambda}=[\bm_{\sigma}]n!\tilde{\rho}(P_{\lambda})
   	=\mu!\sum_{T\in \mathrm{SSYT}(\lambda,\mu)}\psi_{T}(t).
   \end{align}
   On the other hand, for any tableau $T\in \mathrm{SSYT}(\lambda,\mu)$ and a permutation $w\in \mathfrak{S}_n$, 
   \begin{align*}
   	\bx_{c(T_{w(1)})}\bx_{c(T_{w(2)})}\cdots \bx_{c(T_{w(n)})}
   \end{align*}
   is a monomial in $\bm_{\sigma}$, where $\sigma$ is the unique set partition determined by the contents of $T$ and the labeling $w$. Since there are exactly $\lambda(\pi)!=\mu!$ permutations of $[n]$ that fix this monomial, we conclude that 
   \begin{align*}
   	[\bm_{\sigma}]\sum_{w\in \mathfrak{S}_n}\sum_{T\in \mathrm{SSYT}(\lambda)}\psi_{T}(t)
   	\bx_{c(T_{w(1)})}
   	\bx_{c(T_{w(2)})}\cdots \bx_{c(T_{w(n)})}=\mu!\sum_{T\in \mathrm{SSYT}(\lambda,\mu)}\psi_{T}(t),
   \end{align*}
   which together with (\ref{E:P2}) proves the claim (\ref{E:tabP}). It follows from (\ref{E:tabP}) and (\ref{E:tabP2}) that 
 \begin{align}\label{E:Lp2}
 	\bP_{\mu}\star \bPP_{\nu}&=\sum_{\delta\in \S_{r,n}}\delta^{-1}\circ (\bP_{\mu}\bPP_{\nu})\\
 	&=\sum_{\delta\in \S_{r,n}}\delta^{-1}\circ (\sum_{v\in \mathfrak{S}_{n-r}}\sum_{P\in \mathrm{SSYT}(\mu)\atop Q\in \mathrm{SSYT}(\nu)}\psi_{P}(t)\psi_{Q}(t)
 	\bx_{c(P_{v(1)})}
 	\cdots \bx_{c(P_{v(n-r)})}\bx_{c(Q_{1})}
 	\cdots \bx_{c(Q_{r})}).\notag
 \end{align}
Returning to the HL functions in Sym, we express both sides of (\ref{E:prodSym}) using the tableau formula (\ref{E:SSYT_P}), thereby obtaining
\begin{align}\label{E:Lp1}
	\sum_{P\in\mathrm{SSYT}(\mu)\atop Q\in \mathrm{SSYT}(\nu)}\psi_{P}(t)\psi_{Q}(t)x^{P+Q}
	=\sum_{\lambda}c_{\mu\nu}^{\lambda}(t)\sum_{T\in \mathrm{SSYT}(\lambda)}\psi_{T}(t)x^{T}.
\end{align}
Each pair $(P,Q)\in \mathrm{SSYT}(\mu)\times\mathrm{SSYT}(\nu)$ on the left-hand side of (\ref{E:Lp1}) determines a unique $T\in \mathrm{SSYT}(\lambda)$ on the right-hand side via the rule of tableau multiplication (see Section 5 of \cite{Fulton97}). For $T\in\mathrm{SSYT}(\lambda)$ corresponding to the pair $(P,Q)$, we specify a labeling of $T$ from given labelings of $P$ and $Q$. That is, we label the boxes of $T$ so that $c(T_i)=c(P_i)$ for $1\le i\le n-r$ and  $c(T_{n-r+i})=c(Q_i)$ for $1\le i\le r$. Then, 
  \begin{align*}
  	\bx_{c(P_{v(1)})}\cdots \bx_{c(P_{v(n-r)})}\bx_{c(Q_{1})}\cdots \bx_{c(Q_{r})}=\bx_{c(T_{v(1)})}
  	\cdots \bx_{c(T_{v(n-r)})}\bx_{c(T_{n-r+1})}\cdots \bx_{c(T_{n})}
  \end{align*}
  for any permutation $v\in \mathfrak{S}_{n-r}$. Substituting the variables $x$ by $\bx$ in (\ref{E:Lp1}) yields that 
  \begin{align}
  	&\quad \sum_{P\in \mathrm{SSYT}(\mu)\atop Q\in \mathrm{SSYT}(\nu)}\psi_{P}(t)\psi_{Q}(t)
  	\bx_{c(P_{v(1)})}\cdots \bx_{c(P_{v(n-r)})}\bx_{c(Q_{1})}\cdots \bx_{c(Q_{r})}\label{E:nc}\\
  	&=\sum_{\lambda}c_{\mu\nu}^{\lambda}(t)\sum_{T\in \mathrm{SSYT}(\lambda)}\psi_{T}(t)\bx_{c(T_{v(1)})}
  	\cdots \bx_{c(T_{v(n-r)})}\bx_{c(T_{n-r+1})}\cdots \bx_{c(T_{n})}.\notag
  \end{align}
   It follows from (\ref{E:coset}) that 
   \begin{align*}
   	&\quad\,\sum_{\delta\in \S_{r,n}}\sum_{v\in \mathfrak{S}_{n-r}}\delta^{-1}\circ (\bx_{c(T_{v(1)})}
   	\bx_{c(T_{v(2)})}\cdots \bx_{c(T_{v(n-r)})}\bx_{c(T_{n-r+1})}\cdots \bx_{c(T_{n})})\\
   	&=\sum_{\delta\in \S_{r,n}}\sum_{v\in \mathfrak{S}_{n-r}}\delta^{-1}\circ (\bx_{c(T_{v(1)})}
   		\bx_{c(T_{v(2)})}\cdots \bx_{c(T_{v(n-r)})}\bx_{c(T_{v(n-r+1)}))}\cdots \bx_{c(T_{v(n)})})\\
   	&=\sum_{\delta\in \S_{r,n}}\sum_{v\in \mathfrak{S}_{n-r}}\bx_{c(T_{\delta v(1)})}
   	\bx_{c(T_{\delta v(2)})}\cdots \bx_{c(T_{\delta v(n-r)})}\bx_{c(T_{\delta v(n-r+1)})}\cdots \bx_{c(T_{\delta v(n)})}\\
   	&=\sum_{w\in \mathfrak{S}_n}\bx_{c(T_{w(1)})}
   	\bx_{c(T_{w(2)})}\cdots \bx_{c(T_{w(n-1)})}\bx_{c(T_{w(n)})}.
   \end{align*}
   Together with (\ref{E:nc}), (\ref{E:Lp2}) and (\ref{E:tabP}) we conclude that $\bP_{\mu}\star \bPP_{\nu}=\sum_{\lambda}c_{\mu\nu}^{\lambda}(t)\bP_{\lambda}$, which is summed over all partitions $\lambda$ such that $|\lambda|=|\mu|+|\nu|$. Taking $t=0$ gives (\ref{E:LRs}), as desired.
   \qed

   \begin{example}\label{Example:1}
   	For $\lambda=(2,1)$ and $r=1$, 
   	\begin{align}
   		\bP_{21}\star \bs_1&=\bP_{21}\bs_1+(14)\bP_{21}\bs_1+(24)\bP_{21}\bs_1+(34)\bP_{21}\bs_1\notag\\
   		&=\bP_{31}+(1+t)\bP_{22}+(1+t)\bP_{211}\label{E:eg2}
   	\end{align}
   	as stated in Theorem \ref{T:pieri1}. This corresponds to the product rule $P_{21}s_1=P_{31}+(1+t)P_{22}+(1+t)P_{211}$ for HL functions $P_{\lambda}(t)$ in $\Sym$. Let $\T=\{\varepsilon, (2\,3\,4), (3\,4), (1\,2\,3\,4)\}$ be a set of coset representatives of $\mathfrak{S}_2$ in $\mathfrak{S}_4$. Then $\T\ne \S_{1,4}=\{\varepsilon, (1\,4), (2\,4), (3\,4)\}$, and one can check that 
   	\begin{align*}
   		\bP_{21}\bs_1+(2\,3\,4)\bP_{21}\bs_1+(3\,4)\bP_{21}\bs_1+(1\,2\,3\,4)\bP_{21}\bs_1
   		=\bP_{21}\star \bs_1.
   	\end{align*}
   	However, this is not true for $\bP_{12/3}\star \bs_1$, namely, 
   	\begin{align*}
   		\bP_{12/3}\bs_1+(2\,3\,4)\bP_{12/3}\bs_1+(3\,4)\bP_{12/3}\bs_1+(1\,2\,3\,4)\bP_{12/3}\bs_1
   		\ne{\bP}_{12/3}\star \bs_1.
   	\end{align*}
   Moreover, even though $[\bP_{13/2/4}](\bP_{21}\star \bs_1)=[\bP_{211}](\bP_{21}\star \bs_1)=1+t$ by (\ref{E:eg2}), its refinement $[\bP_{13/2/4}](\bP_{12/3}\star \bs_1)\not \in \mathbb{N}[t]$. To be precise,
   	(\ref{E:eg1}) implies that 
   	\begin{align*}
   		\bP_{12/3}\star \bs_1&=\bP_{124/3}+\frac{1}{3}\bP_{123/4}+\frac{t+2}{3}\bP_{12/34}+\frac{2t+1}{6}(\bP_{13/24}+\bP_{14/23})\\
   		&\quad+\frac{4t+2}{3}P_{12/3/4}-\frac{t-1}{6}(\bP_{13/2/4}+\bP_{23/1/4})
   		+\frac{t+1}{2}(\bP_{14/2/3}+\bP_{24/1/3}),
   	\end{align*}
   	and clearly $[\bP_{13/2/4}](\bP_{12/3}\star \bs_1)\not \in \mathbb{N}[t]$.
   \end{example}
   
  	\section{Preliminaries on Quasisymmetric functions}\label{S:5}
  	We begin by reviewing the combinatorial concepts required for our work, before introducing the Hopf algebras and quasisymmetric functions of study. Some definitions are omitted for the sake of brevity, and we refer the reader to the papers \cite{HHL08,HLMvW10,HLMvW11}.
  	\subsection{Set compositions}
  	
  	A \emph{set composition} $\vartheta=\vartheta_1/\!/\cdots /\!/\vartheta_{l}$ of $[n]$ is a sequence of disjoint nonempty subsets $\vartheta_1,\ldots,\vartheta_{l}$ of $[n]$ such that $\cup_{i=1}^{l}\vartheta_i=[n]$, denoted by $\vartheta\vDash [n]$. Each $\vartheta_i$ is called a {\em block} of $\vartheta$, $\ell(\vartheta)=l$ is called the length of $\vartheta$, and $|\vartheta|=n$ is called the size of $\vartheta$. A set composition is also called an ordered set partition and we usually write the elements of each block in increasing order.
  	Let $\alpha(\vartheta)=(|\vartheta_1|, \dots, |\vartheta_l|)$ be the composition determined by the block-sizes of $\vartheta$, and let $\pi(\vartheta)$ be the set partition formed by the blocks of $\vartheta$.  %For two set compositions $\vartheta$ and $\eta$, suppose that after removing all double slashes $/\!/$ of $\vartheta$ and $\eta$, the sequences of numbers in $\vartheta$ and $\eta$ are identical. We then say $\vartheta$ \emph{corrupts} $\eta$ if  $\vartheta$ is  $\eta$ with some bars removed, and say that $\eta$ \emph{reforms} $\vartheta$ if $\vartheta$ is $\eta$ with some bars added. 
  	For a composition $\alpha=(\alpha_1,\ldots,\alpha_{n})$, let $\vartheta=\vartheta_1/\!/\cdots /\!/\vartheta_{\ell(\vartheta)}=\Delta(\alpha)$ be the set composition of $[n]$ such that $i\in \vartheta_j$ if and only if $j=|\{\alpha_r:\alpha_r<\alpha_i\}|+1$. 
  	
  	\begin{example}\label{Example:setcomp}
  		Let $\vartheta=2\,7/\!/1\,3/\!/4/\!/5\,6\vDash [7]$, then $\alpha(\vartheta)=(2,2,1,2)$ and $\pi(\vartheta)=13/27/56/4$. Furthermore, $\Delta(2,1,4,2)=2/\!/14/\!/3$.
  	\end{example}
  
   Given a composition $\alpha$ of $n$, let $\langle\alpha\rangle=\vartheta$ be a set composition of $[n]$ such that $\pi(\vartheta)=[\lambda(\alpha)]$, $\alpha(\vartheta)=\alpha$, and the relative order of blocks of equal size is the same in $[\lambda(\alpha)]$ and $\vartheta$. Moreover, let $\delta_{\vartheta}=\delta_{\pi(\vartheta)}$, then $\vartheta=\delta_{\vartheta}\langle\alpha\rangle$. 
   \begin{example}
   $\langle 2,1,2,3\rangle=45/\!/8/\!/67/\!/123$, and $\vartheta=15/\!/4/\!/68/\!/237=\delta_{\vartheta}\langle 2,1,2,3\rangle$ where $\delta_{\vartheta}=23715684$ written in one-line notation.
   \end{example}

    \subsection{Quasisymmetric functions in commuting variables}
   The Hopf algebra of quasisymmetric functions in commuting variables (QSym) first appeared in the paper by Stanley \cite{Stanley:72} in 1972 and was formally introduced by Gessel \cite{Gessel:84} in 1984. The algebra QSym contains Sym as a subalgebra and it has a natural duality with the Hopf algebra of noncommutative symmetric functions (NSym) developed by Gelfand {\em et al.} \cite{GKLLRT:84}.

   A quasisymmetric function in commuting variables is a bounded degree formal power series $f\in \mathbb{Q}[[X]]=\mathbb{Q}[[x_1,x_2,\ldots]]$ such that for all $k\in \mathbb{N}^+$ and $i_1<i_2<\cdots <i_k$, 
    \begin{align*}
    	[x_{i_1}^{\alpha_1}x_{i_2}^{\alpha_2}\cdots x_{i_k}^{\alpha_k}]f=[x_{1}^{\alpha_1}x_{2}^{\alpha_2}\cdots x_{k}^{\alpha_k}]f
    \end{align*}
    for all compositions $\alpha=(\alpha_1,\alpha_2,\cdots \alpha_k)$. The set of all quasisymmetric functions in commuting variables forms a graded algebra
\begin{align*}
	\mathrm{QSym}=\mathrm{QSym}^0 \oplus \mathrm{QSym}^1 \oplus \cdots \subseteq \mathbb{Q}[[X]],
\end{align*}
 where $\mathrm{QSym}^0=\mathrm{span}\{1\}$ and the $n$th graded $\mathrm{QSym}^n$ for $n\ge 1$ has the bases 
 \begin{align*}
 	\mathrm{Sym}^n=\mathrm{span}\{M_{\alpha}:\alpha\vDash n\}=\mathrm{span}\{S_{\alpha}:\alpha\vDash n\}
 \end{align*} 
 and they are defined as follows, given a composition $\alpha=(\alpha_1,\alpha_2,\ldots,\alpha_{k})\vDash n$. Note that Sym is a subalgebra of QSym and $\mathrm{Sym}^n=\mathrm{Sym}\cap \mathrm{QSym}^n$.
 
 The {\em monomial quasisymmetric function} $M_{\alpha}$ is given by 
 \begin{align*}
 	M_\alpha=\sum_{i_1<i_2<\cdots <i_k} x_{i_1}^{\alpha_1}x_{i_2}^{\alpha_2}\cdots x_{i_{k}}^{\alpha_{k}}.
 \end{align*}
 The {\em quasisymmetric Schur function} $S_{\alpha}$ was introduced by Haglund, Luoto, Mason and van Willigenburg \cite{HLMvW11} as an interpolation between Demazure atoms and Schur functions, much like the role that quasisymmetric functions play between non-symmetric and symmetric functions. Let $\alpha$ and $\beta$ be compositions. Then 
\begin{align}\label{E:quasiS}
	S_{\alpha}=\sum_{\beta}K_{\alpha\beta}M_{\beta}
\end{align}
 where $K_{\alpha\beta}$ is the number of semistandard composition tableaux (SSCs) of shape $\alpha$ and weight $\beta$; see \cite[Theorem 6.1]{HLMvW11}. Note that ComT in \cite{HLMvW11} is equivalent to SSC in \cite{HLMvW10}. For the precise definition of SSCs and Demazure atoms, we refer the reader to the papers \cite[Section 4]{HLMvW11} and \cite[Section 5]{HLMvW10}. 
 
 The integer $K_{\alpha\beta}$ in (\ref{E:quasiS}) is called {\em quasisymmetric Kostka coefficient} and its relation to the Kostka numbers is given by
\begin{align}\label{E:Kost}
	K_{\lambda\mu}=\sum_{\lambda(\alpha)=\lambda}K_{\alpha\gamma}
\end{align}
where $\gamma$ is an arbitrary composition with $\lambda(\gamma)=\mu$. 

The monomial quasisymmetric function $M_{\alpha}$ and the quasisymmetric Schur function $S_{\alpha}$ naturally refine their symmetric counterparts $m_{\lambda}$ and $s_{\lambda}$, respectively. That is, 
\begin{align}\label{E:sS}
	s_{\lambda}=\sum_{\lambda(\alpha)=\lambda}S_{\alpha}\quad \mbox{ and }\quad 
	m_{\lambda}=\sum_{\lambda(\alpha)=\lambda}M_{\alpha}.
\end{align}
\begin{example}
$s_{21}=S_{12}+S_{21}$, where $S_{12}=M_{12}+M_{111}$ and $S_{21}=M_{21}+M_{111}$. This agrees with $s_{21}=m_{21}+2m_{111}=(M_{21}+M_{12})+2M_{111}$.
\end{example}
The quasisymmetric HL function introduced by Haglund, Luoto, Mason and van Willigenburg generalizes the formula (\ref{E:sS}) by employing the combinatorial formula for non-symmetric Macdonald polynomials and its relation to symmetric Macdonald polynomials. 

The quasisymmetric HL function $\CMcal{L}_{\alpha}(X_n;t)$ in commuting variables $X_n=(x_1,\ldots,x_n)$ for a composition $\alpha$ with $\ell(\alpha)\le n$ is defined by
\begin{align*}
	\CMcal{L}_{\alpha}(t)=\CMcal{L}_{\alpha}(X_n;t)&=\sum_{\gamma^+=\alpha}E_{\gamma}(X_n;0,t),
\end{align*}
where $E_{\gamma}(X_n;q,t)$ is the non-symmetric Macdonald polynomial. By the combinatorial formula for $E_{\gamma}(X_n;q,t)$ due to Haglund, Haiman and Loehr \cite[Theorem 3.5.1]{HHL08}, it is immediate that
\begin{align}\label{E:quasiL}
	\CMcal{L}_{\alpha}(t)=\sum_{\gamma^+=\alpha}\sum_{\substack{T\in\textrm{NAT}(\gamma)\\ \maj(\hat{T})=0}}t^{\mathsf{coinv}(\hat{T})}(1-t)^{\kappa(T)}x^{T}
\end{align}
where $\textrm{NAT}(\gamma)$ denotes the set of  non-attacking fillings of column diagram $\gamma$, and $\kappa(T)$ counts the cells $u\in \gamma$ for which $\hat{T}(u)\ne \hat{T}(\mathrm{west}(u))$ in the augmented filling $\hat{T}$ of $T$; see \cite[Equation (7.9)]{HLMvW11}. Specializing both sides of (\ref{E:quasiL}) at $t=0$ and $t=1$, we have 
\begin{align}\label{E:L01}
	\CMcal{L}_{\alpha}(0)=S_{\alpha}\quad \mbox{ and }\quad \CMcal{L}_{\alpha}(1)=M_{\alpha}.
\end{align}
Since the symmetric Macdonald polynomial is a linear combination of the non-symmetric Macdonald polynomials over the field $\mathbb{Q}(q,t)$ according to \cite[Proposition 5.3.1]{HHL08}, \cite{Mac952} and \cite[Lemma 2.5 (a)]{Marshall: 99}, and the HL function is a specialization of symmetric Macdonald polynomial at $q=0$, it follows that 
\begin{align}\label{E:plam}
	P_{\lambda}(t)=\sum_{\lambda(\alpha)=\lambda}\CMcal{L}_{\alpha}(t)=\sum_{\lambda(\gamma^+)=\lambda}E_{\gamma}(X_n;0,t),
\end{align}
as a generalization of (\ref{E:sS}). Meanwhile, (\ref{E:plam}) also implies an extension of (\ref{E:Kost}). Note that $\CMcal{L}_{\alpha}(X;t)$ is quasisymmetric as stated in \cite[Proposition 7.1]{HLMvW11}, thus it can be expanded in the monomial basis of QSym, resulting in that 
\begin{align}\label{E:monoL}
	\CMcal{L}_{\alpha}(t)=\sum_{\beta}d_{\alpha\beta}(t)M_{\beta}=M_{\alpha}+\sum_{\lambda(\beta)<\lambda(\alpha)}d_{\alpha\beta}(t)M_{\beta}
\end{align}
for some $d_{\alpha\beta}(t)\in \mathbb{Z}[t]$ such that $d_{\alpha\beta}(0)=K_{\alpha\beta}$ and $d_{\alpha\beta}(1)=\delta_{\alpha\beta}$ by (\ref{E:quasiS}) and (\ref{E:L01}). It follows from (\ref{E:quasiL}) that if $d_{\alpha\beta}(t)\ne 0$, then $\lambda(\beta)<\lambda(\alpha)$ or ($\alpha=\beta$ and $d_{\alpha\alpha}(t)=1$).

%Moreover, this implies that for any permutation $\delta\in \mathfrak{S}_n$, 
%\begin{align*}
%	\delta\circ \CMcal{L}_{\alpha}(t)=\CMcal{L}_{\delta\alpha}(t).
%\end{align*}
Combining (\ref{E:monoL}) with (\ref{E:Pm11}) and the second equality of (\ref{E:sS}), we find that 
\begin{align*}
	P_{\lambda}(t)=\sum_{\mu\le\lambda}c_{\lambda\mu}(t)m_{\mu}&=\sum_{\mu\le \lambda}c_{\lambda\mu}(t)\sum_{\lambda(\gamma)=\mu}M_{\gamma}
	=\sum_{\lambda(\alpha)=\lambda}\sum_{\gamma}d_{\alpha\gamma}(t)M_{\gamma},
\end{align*}
implying that 
\begin{align*}
	\sum_{\lambda(\alpha)=\lambda}d_{\alpha\gamma}(t)=c_{\lambda\mu}(t)
\end{align*}
for any composition $\gamma$ with $\lambda(\gamma)=\mu$. This is an extension of (\ref{E:Kost}). 
\begin{example}
	$\CMcal{L}_{21}(t)=M_{21}+(1-t)M_{111}$ and  $\CMcal{L}_{12}(t)=M_{12}+(1-t^2)M_{111}$. Therefore,
	$$P_{21}(t)=\CMcal{L}_{21}(t)+\CMcal{L}_{12}(t)=m_{21}+(2-t-t^2)m_{111}.$$
\end{example}

\begin{remark}
	Hivert introduced the quasi-symmetric polynomial $G_{\alpha}(X;t)$ using divided difference operators \cite[Definition 6.1]{Hivert:00}, which includes the fundamental quasisymmetric function $F_{\alpha}(X)=G_{\alpha}(X;0)$ and the monomial quasisymmetric function $M_{\alpha}(X)=G_{\alpha}(X;1)$. 
	The polynomial $G_{\alpha}(X;t)$ exhibits interesting properties, yet it is not a refinement of the HL function $P_{\lambda}(X;t)$. That is,  for $t=0$, 
	\begin{align*}
		s_{\lambda}=P_{\lambda}(X;0)\ne \sum_{\lambda(\alpha)=\lambda}G_{\alpha}(X;0)=\sum_{\lambda(\alpha)=\lambda}F_{\alpha}.
	\end{align*}
\end{remark}

\subsection{Quasisymmetric functions in noncommuting variables}
The Hopf algebra of quasisymmetric functions in noncommuting variables (NCQSym) appeared in the thesis of Hivert \cite{Hi04} in 2004. Subsequently, Bergeron and Zabrocki \cite{BZ:09} proved that NCSym and NCQSym are free and cofree. Lazzeroni \cite{Laz:23} introduced quasisymmetric power-sum bases in NCQSym  that are indexed by set compositions.

The quasisymmetric functions in NCQSym are defined by relating set compositions to compositions via the map $\Delta$. A quasisymmetric function in noncommuting variables is a bounded degree formal power series $f\in \mathbb{Q}[[\bx]]=\mathbb{Q}[[\bx_1,\bx_2,\ldots]]$ such that for all compositions $\alpha=(\alpha_1,\ldots,\alpha_n)$ and $\beta=(\beta_1,\ldots,\beta_n)$ with $\Delta(\alpha)=\Delta(\beta)$, 
 \begin{align*}
 	[\bx_{\alpha_1}\bx_{\alpha_2}\cdots \bx_{\alpha_k}]f=[\bx_{\beta_1}\bx_{\beta_1}\cdots \bx_{\beta_n}]f.
 \end{align*}
 Define the graded Hopf algebra
\begin{align*}
	\mathrm{NCQSym}=\mathrm{NCQSym}^0 \oplus \mathrm{NCQSym}^1 \oplus \cdots \subseteq \mathbb{Q}[[\bx]],
\end{align*}
where $\mathrm{NCQSym}^0=\mathrm{span}\{1\}$ and the $n$th graded $\mathrm{NCQSym}^n$ for $n\ge 1$ has the monomial basis
\begin{align*}
	\mathrm{NCSym}^n=\mathrm{span}\{\bM_{\vartheta}:\vartheta\vDash [n]\}.
\end{align*} 
The {\em monomial quasisymmetric function} $\bM_{\vartheta}$ for a set composition $\vartheta\vDash [n]$ in NCQSym is 
\begin{align*}
	\bM_\vartheta=\sum_{(i_1, \dots, i_n)} \bx_{i_1}\cdots \bx_{i_n}
\end{align*}
summed over all sequences $(i_1, \dots, i_n)$ of positive integers such that $i_a=i_b$ if and only if $a\sim_{\pi(\vartheta)} b$, and $i_a<i_b$ whenever the block containing $a$ precedes the block containing $b$ in the linear order of blocks in $\vartheta$. By definition, for a set partition $\pi$,
\begin{align*}
	\bm_{\pi}=\sum_{\pi(\vartheta)=\pi} \bM_{\vartheta},
\end{align*}
which behaves like the monomial symmetric functions in Sym. 
\begin{example} For $\pi=13/2$, we have $\bm_{13/2}=\bM_{13/\!/2}+\bM_{2/\!/13}$ where 
	\begin{align*}
		\bM_{13/\!/2}&=\sum_{i<j}\bx_i\bx_j\bx_i=\bx_1\bx_2\bx_1+\bx_1\bx_3\bx_1+\bx_2\bx_3\bx_2+\cdots, \\
		\bM_{2/\!/13}&=\sum_{i>j}\bx_i\bx_j\bx_i=\bx_2\bx_1\bx_2+\bx_3\bx_1\bx_3+\bx_3\bx_2\bx_3+\cdots.
	\end{align*}
\end{example}
\subsection{Lifted and non-lifted quasisymmetric HL functions}
We now transform the quasisymmetric HL function $\CMcal{L}_{\alpha}(t)$ in QSym into a polynomial in NCQSym by extending the domain of the lifting map $\tilde{\rho}$ in subsection \ref{ss:lift}. By the projection map $\rho$, we obtain
\begin{align*}
	\rho(\bM_{\vartheta})=M_{\alpha(\vartheta)}.
\end{align*}
The lifting map $\tilde{\rho}$ is defined to be the right inverse of $\rho$ such that $\rho\tilde{\rho}=\varepsilon$ in QSym (not only in Sym). Equivalently, the lifting map $\tilde{\rho}:\mathrm{QSym}\rightarrow \mathrm{NCQSym}$ is given by
   \begin{align}\label{E:lift2}
   	\tilde{\rho}(M_{\alpha})=\frac{\alpha!}{|\alpha|!}\sum_{\alpha(\vartheta)=\alpha}\bM_{\vartheta},
   \end{align}
   and then extended linearly. We will verify that $\rho\tilde{\rho}=\varepsilon$ in QSym. 
   Let  $\vartheta$ be a set composition with $\alpha=\alpha(\vartheta)$. Since there are exactly $|\alpha|!/(\alpha!\alpha^!)$ set partitions $\pi$ satisfying $\lambda(\pi)=\lambda(\alpha)$, and there are $\alpha^!$ to rearrange the blocks of $\pi$ of equal size, we find that 
   \begin{align*}
   	|\{\vartheta\vDash [n]:\alpha(\vartheta)=\alpha\}|=\frac{|\alpha|!}{\alpha!\alpha^!}\alpha^!=\frac{|\alpha|!}{\alpha!}.
   \end{align*}
   It follows that 
    \begin{align*}
   	\rho\tilde{\rho}(M_{\alpha})=\frac{\alpha!}{|\alpha|!}M_{\alpha}\sum_{\alpha(\vartheta)=\alpha}1=M_{\alpha},
   \end{align*}
   namely, $\rho\tilde{\rho}=\varepsilon$ is the identity map in QSym. 

  The {\em lifted quasisymmetric HL function} $\bL_{\alpha}(\bX_n;t)$ for a composition $\alpha\vDash n$ in noncommuting variables $\bX_n=(\bx_1,\cdots,\bx_n)$ is defined by applying the lifting map $\tilde{\rho}$ to the quasisymmetric HL function in (\ref{E:monoL}). 
   \begin{align*}
   	\bL_{\alpha}(t)&=\bL_{\alpha}(\bX_n;t)=n!\tilde{\rho}(\CMcal{L}_{\alpha}(t))=\sum_{\beta}\beta!d_{\alpha\beta}(t)\sum_{\alpha(\vartheta)=\beta}\bM_{\vartheta},\\
   	&=\alpha!\sum_{\alpha(\vartheta)=\alpha}\bM_{\vartheta}+\sum_{\lambda(\beta)<\lambda(\alpha)}\beta!d_{\alpha\beta}(t)\sum_{\alpha(\vartheta)=\beta}\bM_{\vartheta}.
   \end{align*}
   Consequently, the lifting map $\tilde{\rho}$ on both sides of (\ref{E:plam}) gives
   \begin{align}\label{E:PL}
   	\bP_{\lambda}(t)=\sum_{\lambda(\alpha)=\lambda}\bL_{\alpha}(t).
   \end{align}
   The {\em lifted quasisymmetric Schur function} is defined to be 
   \begin{align}\label{E:liftQS}
   	\bS_{\alpha}=\bL_{\alpha}(0)=\alpha!\sum_{\alpha(\vartheta)=\alpha}\bM_{\vartheta}+\sum_{\lambda(\beta)<\lambda(\alpha)}\beta!K_{\alpha\beta}\sum_{\alpha(\vartheta)=\beta}\bM_{\vartheta}.
   \end{align}
   Therefore, $\bL_{\alpha}(t)$ unifies $\bS_{\alpha}$ and $\alpha!\bM_{\alpha}=\bL_{\alpha}(1)$. 
   \begin{example}\label{Example:L21}
   	$\bP_{21}(t)=\bL_{21}(t)+\bL_{12}(t)=2(\bm_{12/3}+\bm_{13/2}+\bm_{23/1})+(2-t-t^2)\bm_{1/2/3}$, where $\bL_{21}(t)$ and $\bL_{12}(t)$ are given by
   	\begin{align*}
    \bL_{21}(t)&=2(\bM_{12/\!/3}+\bM_{13/\!/2}+\bM_{23/\!/1})+(1-t)(\bM_{1/\!/2/\!/3}+\bM_{1/\!/3/\!/2})\\
    &\qquad+(1-t)(\bM_{2/\!/1/\!/3}+\bM_{2/\!/3/\!/1}+\bM_{3/\!/1/\!/2}+\bM_{3/\!/2/\!/1}),\\
    \bL_{12}(t)&=2(\bM_{3/\!/12}+\bM_{2/\!/13}+\bM_{1/\!/23})+(1-t^2)(\bM_{1/\!/2/\!/3}+\bM_{1/\!/3/\!/2})\\
    &\qquad+(1-t^2)(\bM_{2/\!/1/\!/3}+\bM_{2/\!/3/\!/1}+\bM_{3/\!/1/\!/2}+\bM_{3/\!/2/\!/1}).
	\end{align*}
   \end{example}
   
   \section{Quasisymmetric Hall-Littlewood functions in NCQSym}\label{S:6}
   In this section, we introduce the quasisymmetric HL function $\bL_{\vartheta}(\bx;t)$ in noncommuting variables $\bx=(\bx_1,\ldots,\bx_n)$ with coefficients in $\mathbb{Z}[t]$ for a set composition $\vartheta$ and present a sequence of its properties, in parallel with Section \ref{S:3}.
   \begin{definition}
   	Let $\vartheta$ be a set composition with $\alpha(\vartheta)=\alpha$ and $\ell(\vartheta)\le n$. The quasisymmetric HL function in noncommuting variables $\bX_n=(\bx_1,\ldots,\bx_n)$ is defined by
   	\begin{align}\label{E:quasiLnc}
   		\bL_{\vartheta}(t)=\bL_{\vartheta}(\bX_n;t)&=\sum_{\beta}d_{\alpha\beta}(t)\sum_{g\in \mathfrak{S}_{\vartheta}}g\circ \bM_{\delta_{\vartheta}\langle \beta\rangle}
   	\end{align}
   where $d_{\alpha\beta}(t)=[M_{\beta}]\CMcal{L}_{\alpha}\in\mathbb{Z}[t]$ and $\mathfrak{S}_\vartheta=\{\delta\in \mathfrak{S}_n \,\vert\,\delta\vartheta=\vartheta\}$ is the stabilizer group of $\vartheta$, and the quasisymmetric Schur function is defined to be $\bS_{\vartheta}=\bL_{\vartheta}(0)$.
   	\end{definition}
   \begin{example}
   		\begin{align*}
   		\bL_{12/\!/3}(t)&=2\bM_{12/\!/3}+(1-t)(\bM_{1/\!/2/\!/3}+\bM_{2/\!/1/\!/3}),\\
   		\bL_{3/\!/12}(t)&=2\bM_{3/\!/12}+(1-t^2)(\bM_{3/\!/1/\!/2}+\bM_{3/\!/2/\!/1}).
   	\end{align*}
   \end{example}
   Since most of the proofs follow the ideas in Sections \ref{S:3} and \ref{S:4}, we only sketch the proofs.
    \begin{theorem}\label{T:quasi-basis} The set $\mathcal{B}_n=\{{\bL}_{\vartheta}(\bX_n;t): \vartheta \vDash [n]\}$ is a $\mathbb{Z}[t]$-basis of quasisymmetric functions of homogeneous degree $n$ in noncommuting variables $\bX_n=(\bx_1,\ldots,\bx_n)$, and it is invariant under any permutation of $[n]$. That is, $w\mathcal{B}_n=\mathcal{B}_n$ for any $w\in  \mathfrak{S}_n$.
   \end{theorem}
   
    \begin{proof}
    	First, note that $w\mathfrak{S}_{\vartheta}w^{-1}=\mathfrak{S}_{\phi}$ and $w\circ \bM_{\vartheta}=\bM_{w\vartheta}$ for any permutation $w\in \mathfrak{S}_n$, and we claim that $\bL_{\phi}(t)=w\bL_{\vartheta}(t)$ if $\phi=w\vartheta$. Following the argument in Lemma \ref{lem:pi_to_sigma} (2), we find
    	\begin{align*}
    		\sum_{g\in \mathfrak{S}_{\phi}}g\circ \bM_{\delta_{\phi}\langle\beta\rangle}=w\sum_{g\in \mathfrak{S}_{\vartheta}}g\circ\bM_{\delta_{\vartheta}\langle\beta\rangle},
    	\end{align*}
        thus implying that $\bL_{\phi}(t)=w\bL_{\vartheta}(t)$, and hence $w\mathcal{B}_n=\mathcal{B}_n$ for all permutations $w\in\mathfrak{S}_n$. It remains to show that $\mathcal{B}_n$ is a $\mathbb{Z}[t]$-basis of NCQSym$^n$. Since $d_{\alpha\beta}(t)\ne 0$ only when $\lambda(\beta)<\lambda(\alpha)$ or $\alpha=\beta$, and in the latter case $d_{\alpha\alpha}(t)=1$, the transition matrix from the $\bL_{\vartheta}(\bx;t)$ to the $\bM$-basis is triangular with nonzero entries $[\bM_{\vartheta}]\bL_{\vartheta}=|\mathfrak{S}_{\vartheta}|=\pi(\vartheta)!$ on the main diagonal. It follows from (\ref{E:quasiLnc}) that $\mathcal{B}_n$ is also a $\mathbb{Z}[t]$-basis of NCQSym$^n$, as desired.         
    \end{proof}
    We next show that the lifted quasisymmetric HL function $\bL_{\alpha}(t)$ is equal to the sum of quasisymmetric HL functions $\bL_{\vartheta}(t)$ in NCQSym. This is a quasisymmetric analogue of the first equality in (\ref{E:refineP}).
    \begin{theorem}\label{E:sumL}
    	For a composition $\alpha$, the lifted quasisymmetric HL function is given by
    	\begin{align}\label{E:L1}
    		\bL_{\alpha}(t)=\sum_{\alpha(\vartheta)=\alpha}\bL_{\vartheta}(t).
    	\end{align}
        Consequently, the lifted HL function is given by
        \begin{align*}
        	\bP_{\lambda}(t)=\sum_{\lambda(\alpha)=\lambda}\bL_{\alpha}(t)
        	=\sum_{\lambda(\alpha(\vartheta))=\lambda
        	}\bL_{\vartheta}(t).
        \end{align*}
        In particular for $t=0$, we have
        \begin{align*}
        	\bS_{\alpha}=\sum_{\alpha(\vartheta)=\alpha}\bS_{\vartheta}\quad\,\mbox{ and }\quad\,
        	\bS_{\lambda}=\sum_{\lambda(\alpha)=\lambda}\bS_{\alpha}=\sum_{\lambda(\alpha(\vartheta))=\lambda}\bS_{\vartheta}.
        \end{align*}     
    \end{theorem}  
   \begin{proof}
   	Following a similar argument to that in Theorem \ref{T:sumP1}, we can show that for a composition $\alpha$ of $n$, 
   	\begin{align*}
   		\sum_{\alpha(\vartheta)=\alpha} \bL_{\vartheta}(t)&=\sum_{w\in \mathfrak{S}_n/\mathfrak{S}_{\langle\alpha\rangle}} w\bL_{\langle\alpha\rangle}(t)
   		=\sum_{w\in \mathfrak{S}_n/\mathfrak{S}_{\langle\alpha\rangle}} \sum_{\beta}d_{\alpha\beta}(t)\sum_{g\in \mathfrak{S}_{\langle\alpha\rangle}}wg\circ \bM_{\langle \beta\rangle}\\
   		&=\sum_{\beta}d_{\alpha\beta}(t)\sum_{w\in \mathfrak{S}_n}w\circ \bM_{\langle \beta\rangle}
   		=\sum_{\beta}d_{\alpha\beta}(t)\beta!\sum_{\alpha(\phi)=\beta}\bM_{\phi}=\bL_{\alpha}(t). 
   	\end{align*}
   Consequently, the statement is proved by using (\ref{E:PL}).
 \end{proof}
 \begin{example} 
	$\bL_{21}(t)=\bL_{12/\!/3}(t)+\bL_{23/\!/1}(t)+\bL_{13/\!/2}(t)$, where 
	\begin{align*}
		\bL_{12/\!/3}(t)&=2\bM_{12/\!/3}+(1-t)(\bM_{1/\!/2/\!/3}+\bM_{2/\!/1/\!/3}),\\
		\bL_{13/\!/2}(t)&=(23)\bL_{12/\!/3}(t),\\
		\bL_{23/\!/1}(t)&=(123)\bL_{12/\!/3}(t),
	\end{align*}
	and it agrees with the formula for $\bL_{21}(t)$ in Example \ref{Example:L21}.
\end{example}
Finally, we discuss the star product of a lifted quasisymmetric function and a non-lifted Schur function in NCQSym. Let us recall the multiplication rule for a quasisymmetric Schur function and a Schur function due to Haglund, Luoto, Mason and van Willigenburg \cite[Theorem 5.1]{HLMvW11}. For a composition $\alpha$ and a partition $\lambda$, 
\begin{align}\label{E:prodQSs}
	S_{\alpha}s_{\lambda}=\sum_{\beta}C_{\alpha\lambda}^{\beta}S_{\beta},
\end{align}
 where the coefficient $C_{\alpha\lambda}^{\beta}$ is the number of Littlewood--Richardson skew composition tableaux (LRCs) with entries in $[n]$ and of shape $\beta/\alpha$ with content $\lambda^*$ (the reverse of $\lambda$). We refer the reader to \cite[Section 5]{HLMvW10} for the precise definition of LRCs. 
 
 We now state the star multiplication rule for a lifted quasisymmetric function and a non-lifted Schur function in noncommuting variables.
 \begin{theorem}\label{T:pieri2}
 	For $n\ge 2$ and $r\in [n-1]$,
 	let $\alpha\vDash (n-r)$ and $\lambda\vdash r$. Then 
 	\begin{align*}
 		\bS_{\alpha}\star \bs_{\lambda}=\sum_{\beta}C_{\alpha\lambda}^{\beta}\bS_{\beta},
 	\end{align*}
  with $C_{\alpha\lambda}^{\beta}=[S_{\beta}](S_{\alpha}s_{\lambda})\in\mathbb{N}$.
 \end{theorem}
 \begin{proof}
 	Much in spirit of the proof of Theorem \ref{T:pieri1}, we employ the tableau formula (\ref{E:quasiS}) for $S_{\alpha}$, and then transform it to the marked tableau formula for $\bS_{\alpha}$. This, together with (\ref{E:prodQSs}) and the construction of the star product, completes the proof.
 \end{proof}
\begin{example}
	Take $\alpha=(1,2)$, $\lambda=(1)$ and $\bX_4=(\bx_1,\bx_2,\bx_3,\bx_4)$, then
	\begin{align*}
		\bS_{12}\star\bs_{1}=\bS_{13}+\bS_{112}+\bS_{121}.
	\end{align*}
For instance, we examine the coefficients $[\bM_{13/\!/24}](\bS_{12}\star\bs_{1})$ and $[\bM_{13/\!/24}](\bS_{13}+\bS_{112}+\bS_{121})$. First, $[\bM_{14/\!/23}](\bS_{12}\bs_1)=[\bM_{34/\!/12}](\bS_{12}\bs_1)=2$. Therefore,
\begin{align*}
	[\bM_{13/\!/24}](\bS_{12}\star\bs_{1})&=[\bM_{13/\!/24}](34)\circ(\bS_{12}\bs_1)+[\bM_{13/\!/24}](14)\circ(\bS_{12}\bs_1)\\
	&=[\bM_{14/\!/23}](\bS_{12}\bs_1)+[\bM_{34/\!/12}](\bS_{12}\bs_1)=4.
\end{align*}
This coincides with $[\bM_{13/\!/24}](\bS_{13}+\bS_{112}+\bS_{121})=[\bM_{13/\!/24}]\bS_{13}=2!2![M_{22}]S_{13}=4$.
\end{example}
			
			%%%%%%%%%%%%%%%%%%%%%%%%%%%%%%%%%%%%%%%%%%%%%%%%%%%%%%%%%%%%%%%%%%%%%%%%%

			\section*{Acknowledgements}
			This work was initiated when both authors visited the University of British Columbia in 2025, and we would like to thank Stephanie van Willigenburg for many helpful discussions.
			The second author is supported by the National Nature Science Foundation of China (NSFC), Projects No. 12201529 and No. 12571358.

		\end{document}